\documentclass[reqno,12pt,twoside]{ip-journal}
\makeatletter
\usepackage{mathdots}
\usepackage[pdfusetitle, colorlinks=false, bookmarks=false]{hyperref}
\hypersetup{pdfauthor={Rongqing Ye, Elad Zelingher},
	pdfsubject={Represention theory},
	pdfkeywords={Exterior square, p-adic groups, Simple supercuspidal representations}}

\usepackage{cleveref}

\numberwithin{equation}{section}
\numberwithin{figure}{section}
\theoremstyle{definition}
\newtheorem{defn}{\protect\definitionname}[section]
\theoremstyle{plain}
\newtheorem{prop}[defn]{\protect\propositionname}
\theoremstyle{plain}
\newtheorem{thm}[defn]{\protect\theoremname}
\theoremstyle{plain}
\newtheorem{lem}[defn]{\protect\lemmaname}
\theoremstyle{remark}
\newtheorem{rem}[defn]{\protect\remarkname}
\theoremstyle{remark}

\theoremstyle{plain}

\theoremstyle{plain}
\newtheorem*{thm*}{\protect\theoremname}
\newtheorem{conj}[defn]{Conjecture}

\makeatother

\providecommand{\claimname}{Claim}
\providecommand{\definitionname}{Definition}
\providecommand{\lemmaname}{Lemma}
\providecommand{\propositionname}{Proposition}
\providecommand{\remarkname}{Remark}
\providecommand{\corollaryname}{Corollary}
\providecommand{\theoremname}{Theorem}

\providecommand{\claimname}{Claim}
\providecommand{\corollaryname}{Corollary}
\providecommand{\definitionname}{Definition}
\providecommand{\lemmaname}{Lemma}
\providecommand{\propositionname}{Proposition}
\providecommand{\remarkname}{Remark}
\providecommand{\theoremname}{Theorem}
\providecommand{\theoremname}{Theorem}

\usepackage{scalerel}
\usepackage{stackengine,wasysym}

\def\undertilde#1{\mathord{\vtop{\ialign{##\crcr
				$\hfil\displaystyle{#1}\hfil$\crcr\noalign{\kern1.5pt\nointerlineskip}
				$\hfil\tilde{}\hfil$\crcr\noalign{\kern1.5pt}}}}}

\newcommand{\fieldCharacter}{\psi}
\newcommand{\UpperTriangularAdditive}{\mathcal{B}}
\newcommand{\weylElement}[1]{w_{#1}}
\newcommand{\permutationMatrix}[1]{\sigma_{#1}}
\newcommand{\oddPermutationMatrix}{\permutationMatrix{2m+1}}
\newcommand{\evenPermutationMatrix}{\permutationMatrix{2m}}

\newcommand{\NilpotentLowerTriangular}{\mathcal{N}^{-}}
\newcommand{\NilpotentUpperTriangular}{\mathcal{N}^{+}}
\newcommand{\JS}[1]{J}
\newcommand{\TwistedJS}[1]{J}
\newcommand{\TwistedDualJS}[1]{\tilde{J}}
\newcommand{\DualJS}[1]{\tilde{J}}
\newcommand{\rowvector}[1]{\varepsilon_{#1}}
\newcommand{\firstrowvector}{\rowvector{1}}
\newcommand{\lastrowvector}{\varepsilon}
\newcommand{\TwistedExteriorSquareLFunction}[3]{L\left(#1, #2, \wedge^2 \otimes #3 \right)}
\newcommand{\exteriorSquareLFunction}[2]{L\left(#1, #2, \wedge^2 \right)}
\newcommand{\ShalikaDiagonalElement}[1]{\begin{pmatrix}
		#1 &\\		
		& #1
\end{pmatrix}}
\newcommand{\ShalikaUnipotentElement}[1]{\begin{pmatrix}
		\IdentityMatrix{m} & #1\\		
		& \IdentityMatrix{m}
\end{pmatrix}}
\newcommand{\SmallShalikaDiagonalElement}[1]{\left(\begin{smallmatrix}
		#1 &\\		
		& #1
	\end{smallmatrix}\right)}

\newcommand{\SmallShalikaUnipotentElement}[1]{\left(\begin{smallmatrix}
		\IdentityMatrix{m} & #1\\		
		& \IdentityMatrix{m}
	\end{smallmatrix}\right)}

\newcommand{\ShalikaDiagonalElementOdd}[1]{\begin{pmatrix}
		#1 & &\\		
		& #1 &\\
		& & 1
\end{pmatrix}}

\newcommand{\ShalikaUnipotentElementOdd}[1]{\begin{pmatrix}
		\IdentityMatrix{m} & #1 &\\		
		& \IdentityMatrix{m} & \\
		& & 1
\end{pmatrix}}

\newcommand{\ShalikaLowerUnipotentElementOdd}[1]{\begin{pmatrix}
		\IdentityMatrix{m} & &\\		
		& \IdentityMatrix{m} & \\
		& #1 & 1
\end{pmatrix}}

\newcommand{\ShalikaUpperRightUnipotentElementOdd}[1]{\begin{pmatrix}
		\IdentityMatrix{m} & & #1\\		
		& \IdentityMatrix{m} & \\
		&  & 1
\end{pmatrix}}

\newcommand{\WhittakerModelOfLocalFieldRepresentation}{\Whittaker \left(\localFieldRepresentation, \fieldCharacter\right)}
\newcommand{\gammaFactorOfLocalField}[1]{\gamma \left(#1, \localFieldRepresentation, \wedge^2, \fieldCharacter \right)}

\newcommand{\unitaryDetCharacter}{\mu}

\newcommand{\TwistedEpsilonFactor}[2]{\epsilon \left(#1, \localFieldRepresentation, \wedge^2 \otimes #2, \fieldCharacter \right)}
\newcommand{\twistedGammaFactor}[3]{\gamma \left(#1, #2, \wedge^2 \otimes #3, \fieldCharacter \right)}

\newcommand{\TwistedEpsilonFactorOfLocalField}[1]{\TwistedEpsilonFactor{s}{\unitaryDetCharacter}}

\newcommand{\twistedGammaFactorOfLocalField}[1]{\twistedGammaFactor{s}{\localFieldRepresentation}{\unitaryDetCharacter}}
\newcommand{\epsilonFactorOfLocalField}[1]{\epsilon \left(#1, \localFieldRepresentation, \wedge^2, \fieldCharacter \right)}
\newcommand{\JSOfLocalFieldRepresentation}[3]{\JS{\localFieldRepresentation} \left(#1,#2, #3, \fieldCharacter \right)}
\newcommand{\TwistedJSOfLocalFieldRepresentation}[3]{\TwistedJS{\localFieldRepresentation}\left(#1,#2, #3, {\unitaryDetCharacter}, {\fieldCharacter} \right)}

\newcommand{\TwistedDualJSOfLocalFieldRepresentation}[3]{\TwistedDualJS
	{\localFieldRepresentation} \left(#1,#2, #3, {\unitaryDetCharacter}, {\fieldCharacter} \right)}
\newcommand{\DualJSOfLocalFieldRepresentation}[3]{\DualJS{\localFieldRepresentation} \left(#1,#2, #3, \fieldCharacter\right)}
\newcommand{\realPartHalfRightPlaneLocalFieldRepresentation}{r_{\localFieldRepresentation, \wedge^2}}
\newcommand{\realPartHalfRightPlaneLocalFieldDualRepresentation}{r_{\Contragradient{\localFieldRepresentation}, \wedge^2}}

\newcommand{\localFieldRepresentationContragradientIndex}{{\Contragradient{\localFieldRepresentation}}}

\newcommand{\localFieldRepresentationIndex}{{\localFieldRepresentation}}

\newcommand{\affineGenericCharacter}{\chi}
\newcommand{\rootOfUniformizer}{\zeta}
\newcommand{\uniformizerValueForTwistedGamma}{\xi}

\newcommand{\proUnipotentRadical}{I^{+}}
\newcommand{\centerTimesProUnipotent}{H}
\newcommand{\modifiedCenterTimesProUnipotent}{H'}
\newcommand{\multiplicativeCharacter}{\omega}

\newcommand{\twistedMultiplicativeExteriorCharacter}{\exteriorCharacter{\left(\multiplicativeCharacter \cdot \unitaryDetCharacter^m \right)}}
\newcommand{\exteriorCharacter}[1]{#1_{\uniformizerValueForTwistedGamma}}

\newcommand{\cyclicGroupGenerator}[1]{g_{#1}}
\newcommand{\simpleRep}{\sigma_{\affineGenericCharacter}^{\rootOfUniformizer}}
\newcommand{\antidiagPermutationMatrix}[1]{w_{#1}}
\newcommand{\antidiagPermutation}[1]{\tau_{#1}}
\newcommand{\uniformizerDiagonalMatrix}[1]{d_{#1}}

\newcommand{\modulooperator}{\operatorname{mod}}
\newcommand{\zIntegers}{\mathbb{Z}}
\newcommand{\rReal}{\mathbb{R}}
\newcommand{\cComplex}{\mathbb{C}}
\newcommand{\multiplicativegroup}[1]{#1^{\ast}}
\newcommand{\RealPart}{\mathrm{Re}}

\newcommand{\Span}{\mathrm{span}}

\newcommand{\conjugate}[1]{\overline{#1}}
\newcommand{\indicatorFunction}[1]{\delta_{#1}}
\newcommand{\isomorphic}{\cong}
\newcommand{\sizeof}[1]{\left|#1\right|}

\newcommand{\standardForm}[2]{\left\langle #1,#2\right\rangle}

\newcommand{\centralCharacter}[1]{\omega_{#1}}

\newcommand{\CompactInd}[3]{\mathrm{ind}_{#1}^{#2}\left(#3\right)}

\newcommand{\UnipotentSubgroup}{N}

\newcommand{\rightHaarMeasureModulus}[1]{\delta_{#1}^{-1}}

\newcommand{\Schwartz}{\mathcal{S}}
\newcommand{\Whittaker}{\mathcal{W}}

\newcommand{\Contragradient}[1]{\widetilde{#1}}

\newcommand{\FourierTransformWithRespectToCharacter}[2]{\mathcal{F}_{#2}#1}

\newcommand{\representationDeclaration}[1]{\left(#1, V_{#1}\right)}

\newcommand{\gammaFactorOfCharacter}[2]{\gamma\left(#1,#2, \fieldCharacter\right)}

\newcommand{\maximalCompactSubgroup}{K}
\newcommand{\borelSubgroup}{B}
\newcommand{\weylGroup}[1]{W_{#1}}

\newcommand{\uniformizer}{\varpi}
\newcommand{\modifiedUniformizer}{t^{-1} \varpi}
\newcommand{\integersring}{\mathfrak{o}}
\newcommand{\maximalideal}{\mathfrak{p}}
\newcommand{\residueField}{\mathfrak{f}}
\newcommand{\localField}{F}
\newcommand{\localFieldRepresentation}{\pi}
\newcommand{\quotientMap}{\nu}

\newcommand{\multiplicativeMeasure}[1]{{d^{\times}{#1}}}
\newcommand{\abs}[1]{\left|#1\right|}
\newcommand{\residueFieldUnipotentSubgroup}[1]{\UnipotentSubgroup_{#1} \left( \residueField \right)}
\newcommand{\residueFieldUnipotentSubgroupDoubleCoset}[2]{\residueFieldUnipotentSubgroup{#1} \antidiagPermutationMatrix{#2} \residueFieldUnipotentSubgroup{#1}}

\newcommand{\rquot}[2]{{#1}\slash{#2}}
\newcommand{\lquot}[2]{{#1}\backslash{#2}}
\newcommand{\grpIndex}[2]{\left[#1:#2\right]}
\newcommand{\grpGeneratedBy}[1]{\left\langle #1 \right\rangle}

\newcommand{\transpose}[1]{\, {}^{t}#1}
\newcommand{\inverseTranspose}[1]{#1^{\iota}}
\newcommand{\IdentityMatrix}[1]{I_{#1}}
\newcommand{\diag}{\mathrm{diag}}

\newcommand{\trace}{\mathrm{tr}}
\newcommand{\GL}[2]{\mathrm{GL}_{#1}\left(#2\right)}
\newcommand{\SquareMat}[2]{M_{#1}\left(#2\right)}
\newcommand{\Mat}[3]{M_{#1 \times #2}\left(#3\right)}

\newcommand{\standardColumnVector}[1]{e_{#1}}
\newcommand{\diagonalSubgroup}{A}
\newcommand{\smallDiagTwo}[2]{\left(\begin{smallmatrix}
		#1 & \\
		& #2
	\end{smallmatrix}\right)}

\newcommand{\diagTwo}[2]{\begin{pmatrix}
		#1 & \\
		& #2
\end{pmatrix}}

\newcommand{\smallAntiDiagTwo}[2]{\left(\begin{smallmatrix}
		& #1 \\
		#2 & 
	\end{smallmatrix}\right)}

\newcommand{\antiDiagTwo}[2]{\begin{pmatrix}
		& #1 \\
		#2 & 
\end{pmatrix}}

\title[Exterior square $\gamma$-factors for simple supercuspidal representations]{Exterior square gamma factors for cuspidal representations of $\mathrm{GL}_n$: simple supercuspidal representations}

\author{Rongqing Ye}
\address{Department of Mathematics, Purdue University, West Lafayette, IN 47907, USA}
\email{ye271@purdue.edu}

\author{Elad Zelingher}
\address{Department of Mathematics, University of Michigan, 1844 East Hall, 530 Church Street, Ann Arbor, MI 48109-1043 USA}
\email{eladz@umich.edu}

\begin{document}

\begin{abstract}
We compute the local twisted exterior square gamma factors for simple supercuspidal representations, using which we prove a local converse theorem for simple supercuspidal representations.
\end{abstract}

\maketitle

\section{Introduction}\label{sec:introduction}

A local conjecture of Jacquet for $\GL{n}{F}$, where $F$ is a local non-Archimedean field, asserts that the structure of an irreducible generic representation can be determined by a family of twisted Rankin-Selberg gamma factors. This conjecture was completely settled independently by Chai \cite{Chai19} and Jacquet-Liu \cite{JacquetLiu16}, using different methods:

\begin{thm}[Chai \cite{Chai19}, Jacquet-Liu \cite{JacquetLiu16}]
    Let $\pi_1$ and $\pi_2$ be irreducible generic representations of $\GL{n}{F}$ sharing the same central character. Suppose for any $1 \le r \le \left\lfloor \frac{n}{2} \right\rfloor$ and for any irreducible generic representation $\tau$ of $\GL{r}{F}$,
    $$\gamma(s, \pi_1 \times \tau, \psi) = \gamma(s, \pi_2 \times \tau, \psi).$$
    Then $\pi_1 \cong \pi_2$.
\end{thm}

The bound $\left\lfloor \frac{n}{2} \right\rfloor$ for $r$ in the theorem can be shown to be sharp by constructing some pairs of generic representations. However, the sharpness of $\left\lfloor \frac{n}{2} \right\rfloor$ is not that obvious if we replace ``generic'' by ``unitarizable supercuspidal'' in the theorem. In the tame case, it is shown in \cite{AdrianLiuStevensTam18} that $\left\lfloor \frac{n}{2} \right\rfloor$ is indeed sharp for unitarizable supercuspidal representations of $\GL{n}{F}$ when $n$ is prime. For some certain families of supercuspidal representations, $\left\lfloor \frac{n}{2} \right\rfloor$ is no longer sharp and the $\GL{1}{F}$ twisted Rankin-Selberg gamma factors might be enough to determine the structures of representations within these families. Such a family of supercuspidal representations can be a family of simple supercuspidal representations, see \cite[Proposition 2.2]{BushnellHenniart14} and \cite[Remark 3.18]{adrian2016some}, and also be a family of level zero supercuspidal representations for certain $n$, see \cite[Section 4.6]{NienZhang18}.

In this paper, we consider another kind of local converse theorems of Ramakrishnan using twisted exterior power gamma factors from \cite{Ramakrishnan94}.

\begin{conj}[Ramakrishnan]\label{conj:converse_theorem_using_exterior_power}
    Let $\pi_1$ and $\pi_2$ be irreducible unitarizable supercuspidal representations of $\GL{n}{F}$ sharing the same central character. Suppose for any character $\chi$ of $\multiplicativegroup{F}$, we have
    $$\gamma(s, \pi_1 \times \chi, \psi) = \gamma(s, \pi_2 \times \chi, \psi),$$
    and
    $$\gamma(s, \pi_1, \wedge^j \otimes \chi, \psi) = \gamma(s, \pi_2, \wedge^j \otimes \chi, \psi),$$
    for any $2 \le j \le \left\lfloor \frac{n}{2} \right\rfloor$. Then $\pi_1 \cong \pi_2$.
\end{conj}

We note here that the condition on sharing the same central character is redundant, since if $\gamma(s, \pi_1 \times \chi, \psi) = \gamma(s, \pi_2 \times \chi, \psi)$ for all characters $\chi$, this guarantees that the representations have the same central character as in \cite[Corollary 2.7]{JiangNienStevens15}. We leave it in the statement of the conjecture as a general requirement for a local converse problem. In fact, in the formulation of the main result \Cref{thm:main_thm}, we will need to assume that the representations in consideration share the same central character.

When $j = 2$, the twisted exterior square gamma factors of irreducible supercuspidal representations of $\GL{n}{F}$ exist due to the work of Jacquet-Shalika \cite{jacquet11exterior} together with Matringe \cite{matringe2014linear} and Cogdell-Matringe \cite{CogdellMatringe15}, or the work of Shahidi \cite{Shahidi90} using the Langlands-Shahidi method. When $j = 3$, Ginzburg and Rallis \cite{GinzburgRallis00} found an integral representation for the automorphic $L$-function $L(s, \pi, \wedge^3 \otimes \chi)$  attached to an irreducible cuspidal automorphic representation $\pi$ of $\GL{6}{\mathbb{A}}$ and a character $\chi$ of $\GL{1}{\mathbb{A}}$ for some adelic ring $\mathbb{A}$. In general, for $j \ge 3$, we don't have an analytic definition for $\gamma(s, \pi, \wedge^j \otimes \chi, \psi)$. Therefore, \Cref{conj:converse_theorem_using_exterior_power} only makes sense for $n = 4, 5$ and possibly $6$ if one can prove local functional equations for the local integrals coming from \cite{GinzburgRallis00}.

Since we have only twisted exterior square gamma factors in general, we want to know which families of supercuspidal representations of $\GL{n}{F}$ satisfy \Cref{conj:converse_theorem_using_exterior_power} when $j = 2$. We show in the paper that \Cref{conj:converse_theorem_using_exterior_power} holds true for simple supercuspidal representations up to a sign as we will explain in the next paragraph. This result is our first step toward \Cref{conj:converse_theorem_using_exterior_power}. We have already seen that $\GL{1}{F}$ twists are enough to distinguish simple supercuspidal representations, see \cite[Proposition 2.2]{BushnellHenniart14} and \cite[Remark 3.18]{adrian2016some}. Thus, \Cref{conj:converse_theorem_using_exterior_power} for simple supercuspidal representations has no context if we still require $\GL{1}{F}$ twists. Therefore, we will drop the assumption on the $\GL{1}{F}$ twists.

Let $\mathfrak{o}$ be the ring of integers of $F$, and $\mathfrak{p} = (\varpi)$ is the maximal prime ideal in $\mathfrak{o}$ generated by a fixed uniformizer $\varpi$. By \cite{knightly2015simple}, if we fix a tamely ramified central character $\omega$, i.e., $\omega$ is trivial on $1+\mathfrak{p}$, there are exactly $n(q-1)$ isomorphism classes of irreducible simple supercuspidal representations of $\GL{n}{F}$, each of which corresponds to a pair $(t_0, \zeta)$, where $t_0 \in \multiplicativegroup{\mathfrak{f}}$ is a non-zero element in the residue field $\mathfrak{f}$ of $F$ and $\zeta$ is an $n$-th root of $\omega(t^{-1}\varpi)$, where $t$ is a lift of $t_0$ to $\multiplicativegroup{\mathfrak{o}}$. The main theorem of the paper is the following:

\begin{thm}\label{thm:main_thm}
    Let $\localFieldRepresentation$ and $\localFieldRepresentation'$ be irreducible simple supercuspidal representations of $\GL{n}{F}$ sharing the same central character $\omega$, such that $\pi$, $\pi'
    $ are associated with the data $(t_0, \rootOfUniformizer)$, $(t_0', \rootOfUniformizer')$ respectively. Assume that
    \begin{enumerate}
        \item $\gcd \left(m-1, q-1\right)=1$ if $n = 2m$,
        \item or $\gcd \left(m, q-1\right)=1$ if $n = 2m+1$.
    \end{enumerate}
    Suppose for every unitary tamely ramified character $\mu$ of $\multiplicativegroup{F}$, we have
    $$\gamma(s, \localFieldRepresentation, \wedge^2 \otimes \mu, \psi) = \gamma(s, \localFieldRepresentation', \wedge^2 \otimes \mu, \psi).$$
    Then $t_0 = t_0'$ and $\rootOfUniformizer = \pm \rootOfUniformizer'$. Moreover, we have $\rootOfUniformizer = \rootOfUniformizer'$ if $n = 2m + 1$ is odd.
\end{thm}

In the case $n = 2m$, we can only show $\rootOfUniformizer$ and $\rootOfUniformizer'$ are equal up to a sign. That is what we mean by saying that \Cref{conj:converse_theorem_using_exterior_power} holds true for simple supercuspidal representations up to a sign. \Cref{thm:main_thm}, as far as we know, is the first result toward \Cref{conj:converse_theorem_using_exterior_power} of Ramakrishnan.

In \Cref{sec:preliminaries}, we recall the definitions of the twisted exterior square gamma factors following \cite{jacquet11exterior,matringe2014linear,CogdellMatringe15}. We then recall some results on simple supercuspidal representations in \cite{knightly2015simple}. More importantly from \cite[Section 3.3]{adrian2016some}, we have explicit Whittaker functions for such simple supercuspidal representations. Using these explicit Whittaker functions, we compute in \Cref{sec:computation} the twisted exterior square gamma factors. Finally in \Cref{section:local-converse-theorem}, we prove our main theorem, \Cref{thm:main_thm}.

\section{Preliminaries and notation}\label{sec:preliminaries}

\subsection{Notation}\label{subsubsection:notations}

Let $\localField$ be a non-archimedean local field. We denote by $\integersring$ its ring of integers, by $\maximalideal$ the unique prime ideal of $\integersring$, by $\residueField = \rquot{\integersring}{\maximalideal}$ its residue field. Denote $q = \sizeof{\residueField}$.

Let $\quotientMap : \integersring \rightarrow \residueField$ be the quotient map. We continue denoting by $\quotientMap$ the maps that $\quotientMap$ induces on various groups, for example $\integersring^m \rightarrow \residueField^m$,  $\SquareMat{m}{\integersring} \rightarrow \SquareMat{m}{\residueField}$, $\GL{m}{\integersring} \rightarrow \GL{m}{\residueField}$ etc.

Let $\uniformizer$ be a uniformizer (a generator of $\maximalideal$). We denote by $\abs{\cdot}$, the absolute value on $\localField$, normalized such that $\abs{\uniformizer} = \frac{1}{q}$.

Let $\fieldCharacter : \localField \rightarrow \multiplicativegroup{\cComplex}$ be a non-trivial additive character with conductor $\maximalideal$, i.e. $\fieldCharacter$ is trivial on $\maximalideal$, but not on $\integersring$.

\subsection{The twisted Jacquet-Shalika integral}

In this section, we define twisted versions of the Jacquet-Shalika integrals, and discuss the functional equations that they satisfy. This will allow us to define the twisted exterior square gamma factor $\twistedGammaFactorOfLocalField{s}$ for a generic representation $\representationDeclaration{\localFieldRepresentation}$ of $\GL{n}{\localField}$ and a unitary character $\unitaryDetCharacter : \multiplicativegroup{\localField} \rightarrow \multiplicativegroup{\cComplex}$. We will need this for our local converse theorem in \Cref{section:local-converse-theorem}.

We denote $\UnipotentSubgroup_m$ the upper unipotent subgroup of $\GL{m}{\localField}$, $\diagonalSubgroup_m$ the diagonal subgroup of $\GL{m}{\localField}$, $\maximalCompactSubgroup_m = \GL{m}{\integersring}$, $\UpperTriangularAdditive_m$ the upper triangular matrix subspace of $\SquareMat{m}{\localField}$, and  $\NilpotentLowerTriangular_m$ the lower triangular nilpotent matrix subspace of $\SquareMat{m}{\localField}$. We have $\lquot{\UpperTriangularAdditive_m}{\SquareMat{m}{\localField}} \isomorphic \NilpotentLowerTriangular_m$.

For the following pairs of groups $A \le B$, we normalize the Haar measure on $B$ so that the compact open subgroup $A$ has measure one: $\integersring \le \localField$, $\multiplicativegroup{\integersring} \le \multiplicativegroup{F}$, $\maximalCompactSubgroup_m = \GL{m}{\integersring} \le \GL{m}{\localField}$, $\NilpotentLowerTriangular_m\left(\integersring\right) \le \NilpotentLowerTriangular_m\left( \localField \right)$.

Recall the Iwasawa decomposition: $\GL{m}{\localField} = \UnipotentSubgroup_m \diagonalSubgroup_m \maximalCompactSubgroup_m$. It follows from this decomposition that for an integrable function $f : \lquot{\UnipotentSubgroup_m}{\GL{m}{\localField}} \rightarrow \cComplex$ we have $$\int_{\lquot{\UnipotentSubgroup_m}{\GL{m}{\localField}}} f\left(g\right) \multiplicativeMeasure{g} = \int_{\maximalCompactSubgroup_m} \int_{\diagonalSubgroup_m} f\left(ak\right) \rightHaarMeasureModulus{\borelSubgroup_m}\left(a\right) \multiplicativeMeasure{a} \multiplicativeMeasure{k},$$
where $\borelSubgroup_m \le \GL{m}{\localField}$ is the Borel subgroup, and for $a = \diag\left(a_1,\dots,a_m\right)$, $\rightHaarMeasureModulus{\borelSubgroup_m}\left(a\right) = \prod_{1 \le i < j \le m} \abs{\frac{a_j}{a_i}}$ is the Haar measure module character.

Let $\representationDeclaration{\localFieldRepresentation}$ be an irreducible generic representation of $\GL{n}{F}$, and denote its Whittaker model with respect to $\fieldCharacter$ by $\WhittakerModelOfLocalFieldRepresentation$. Let $\unitaryDetCharacter : \multiplicativegroup{\localField} \rightarrow \multiplicativegroup{\cComplex}$ be a unitary character. We now define the twisted Jacquet-Shalika integrals and their duals. These initially should be thought as formal integrals. We discuss their convergence domains later and explain how to interpret them for arbitrary $s \in \cComplex$.

We have a map $\WhittakerModelOfLocalFieldRepresentation \rightarrow \Whittaker \left( \Contragradient{\localFieldRepresentation}, \fieldCharacter^{-1} \right)$, denoted by $W \mapsto \Contragradient{W}$, where $\Contragradient{\localFieldRepresentation}$ is the contragredient representation, and $\Contragradient{W}$ is given by $\Contragradient{W} \left(g\right) = W \left( \weylElement{n} \inverseTranspose{g} \right)$, where $\weylElement{n} = \left(\begin{smallmatrix}
& & 1\\
& \iddots &\\
1
\end{smallmatrix}\right)$ and $\inverseTranspose{g} = \transpose{g}^{-1}$.

Denote by $\Schwartz\left(\localField^m\right)$ the space of Schwartz functions $\phi : \localField^m \rightarrow \cComplex$, that is the space of locally constant functions with compact support.

Suppose $n = 2m$. We define for $s \in \cComplex$, $W \in \WhittakerModelOfLocalFieldRepresentation$, $\phi \in \Schwartz \left(\localField^m \right)$
\begin{align*}
	\TwistedJSOfLocalFieldRepresentation{s}{W}{\phi} =  \int_{\lquot{\UnipotentSubgroup_m}{\GL{m}{\localField}}} \int_{\lquot{\UpperTriangularAdditive_m}{\SquareMat{m}{\localField}}} & W \left( \evenPermutationMatrix \ShalikaUnipotentElement{X} \ShalikaDiagonalElement{g} \right) \fieldCharacter \left(- \trace X \right) \\
	& \cdot \abs{\det g}^s \unitaryDetCharacter \left( \det g \right) \phi \left( \lastrowvector g \right) dX \multiplicativeMeasure{g},
\end{align*}
where $\lastrowvector = \rowvector{m} = \begin{pmatrix}
0 & \dots & 0 & 1
\end{pmatrix} \in \localField^m$, and $\evenPermutationMatrix$ is the column permutation matrix corresponding to the permutation $$ \begin{pmatrix}
1 & 2 & \dots & m & \mid & m + 1 & m + 2 & \dots & 2m\\
1 & 3 & \dots & 2m - 1 & \mid & 2 & 4 & \dots & 2m
\end{pmatrix},$$ i.e., $$\evenPermutationMatrix = \begin{pmatrix}
\standardColumnVector{1} & \standardColumnVector{3} & \dots & \standardColumnVector{2m-1} & \standardColumnVector{2} & \standardColumnVector{4} & \dots & \standardColumnVector{2m}
\end{pmatrix},$$
where $\standardColumnVector{i}$ is the $i$-th standard column vector, for $1 \le i \le 2m$. 
In this case, we define the dual Jacquet-Shalika as $$\TwistedDualJSOfLocalFieldRepresentation{s}{W}{\phi} = \TwistedJS{\Contragradient{\localFieldRepresentation}}\left(1-s, \Contragradient{\localFieldRepresentation} \begin{pmatrix}
& \IdentityMatrix{m}\\
\IdentityMatrix{m} &
\end{pmatrix} \Contragradient{W}, \FourierTransformWithRespectToCharacter{\phi}{\fieldCharacter}, {\unitaryDetCharacter^{-1}}, {\fieldCharacter^{-1}} \right),$$ where $J$ on the right hand side is the Jacquet-Shalika integral for $\Contragradient{\pi}$, and $$\FourierTransformWithRespectToCharacter{\phi}{\fieldCharacter} \left(y\right) = q^{\frac{m}{2}} \int_{\localField^m} \phi \left( x \right) \fieldCharacter \left( \standardForm{x}{y} \right) dx$$ is the Fourier transform, normalized such that it is self dual (here $\standardForm{\cdot}{\cdot}$ is the standard bilinear form on $\localField^m$).

Suppose $n = 2m + 1$. We define for $s \in \cComplex$, $W \in \WhittakerModelOfLocalFieldRepresentation$, $\phi \in \Schwartz \left(\localField^m \right)$

\begin{align*}
	\TwistedJSOfLocalFieldRepresentation{s}{W}{\phi}= & \int_{\lquot{\UnipotentSubgroup_m}{\GL{m}{\localField}}} \int_{\Mat{1}{m}{\localField}} \int_{\lquot{\UpperTriangularAdditive_m}{\SquareMat{m}{\localField}}} \\
	& \cdot W \left( \oddPermutationMatrix \ShalikaUnipotentElementOdd{X} \ShalikaDiagonalElementOdd{g} \ShalikaLowerUnipotentElementOdd{Z} \right) \\
	& \cdot \fieldCharacter \left(- \trace X \right) \abs{\det g}^{s - 1} \unitaryDetCharacter \left( \det g \right) \phi \left( Z \right) dX dZ \multiplicativeMeasure{g},
\end{align*}
where $\oddPermutationMatrix$ is the column permutation matrix corresponding to the permutation $$ \begin{pmatrix}
1 & 2 & \dots & m & \mid & m + 1 & m + 2 & \dots & 2m & \mid & 2m+1 \\
1 & 3 & \dots & 2m - 1 & \mid & 2 & 4 & \dots & 2m & \mid & 2m+1
\end{pmatrix},$$
i.e. $$\oddPermutationMatrix = \begin{pmatrix}
\standardColumnVector{1} & \standardColumnVector{3} & \dots & \standardColumnVector{2m-1} & \standardColumnVector{2} & \standardColumnVector{4} & \dots & \standardColumnVector{2m} & \standardColumnVector{2m+1}
\end{pmatrix}.$$
In this case, we define the dual Jacquet-Shalika as $$\TwistedDualJSOfLocalFieldRepresentation{s}{W}{\phi} = \TwistedJS{\Contragradient{\localFieldRepresentation}}\left(1-s, \Contragradient{\localFieldRepresentation} \begin{pmatrix}
& \IdentityMatrix{m} &\\
\IdentityMatrix{m} & & \\
& & 1
\end{pmatrix} \Contragradient{W}, \FourierTransformWithRespectToCharacter{\phi}{\fieldCharacter}, {\unitaryDetCharacter^{-1}}, {\fieldCharacter^{-1}} \right).$$

In both cases, we denote $\JSOfLocalFieldRepresentation{s}{W}{\phi} = \TwistedJS{\localFieldRepresentation}\left(s,W,\phi,{1},{\fieldCharacter}\right)$ and $\DualJSOfLocalFieldRepresentation{s}{W}{\phi} = \TwistedDualJS{\localFieldRepresentation}\left(s,W,\phi,{1},{\fieldCharacter}\right)$, where $1$ denotes the trivial character $\multiplicativegroup{\localField} \rightarrow \multiplicativegroup{\cComplex}$.

The definitions of the twisted Jacquet-Shalika integrals are motivated from \cite{jacquet11exterior,matringe2014linear,CogdellMatringe15}. We now list properties of the twisted Jacquet-Shalika integrals, some of which are only proven for the (untwisted) Jacquet-Shalika integrals in the literature.

From now and on suppose $n = 2m$ or $n = 2m + 1$.

\begin{thm}[{\cite[Section 7, Proposition 1, Section 9, Proposition 3]{jacquet11exterior}}]
	There exists $\realPartHalfRightPlaneLocalFieldRepresentation \in \rReal$, such that for every $s \in \cComplex$ with $\RealPart \left(s\right) > \realPartHalfRightPlaneLocalFieldRepresentation$, $W \in \WhittakerModelOfLocalFieldRepresentation$ and $\phi \in \Schwartz \left( \localField^m \right)$, the integral $\TwistedJSOfLocalFieldRepresentation{s}{W}{\phi}$ converges absolutely.
\end{thm}
Similarly, there exists a half left plane $\RealPart\left(s\right) < \realPartHalfRightPlaneLocalFieldDualRepresentation$ (where $\realPartHalfRightPlaneLocalFieldDualRepresentation = 1 - \realPartHalfRightPlaneLocalFieldRepresentation$), in which the dual twisted Jacquet-Shalika integrals $\TwistedDualJSOfLocalFieldRepresentation{s}{W}{\phi}$ converge absolutely for every $W, \phi$.

\begin{thm}[{\cite[Theorem 2.3]{jo2018derivatives}}, {\cite[Lemma 3.1]{CogdellMatringe15}}]
	For fixed $W \in \WhittakerModelOfLocalFieldRepresentation$, $\phi \in \Schwartz \left( \localField^m \right)$. the map $s \mapsto \TwistedJSOfLocalFieldRepresentation{s}{W}{\phi}$ for $s \in \cComplex$ with $\RealPart\left(s\right) > \realPartHalfRightPlaneLocalFieldRepresentation$ results in an element of $\cComplex \left(q^{-s}\right)$, that is a rational function in the variable $q^{-s}$, and therefore has a meromorphic continuation to the entire plane, which we continue to denote as $\TwistedJSOfLocalFieldRepresentation{s}{W}{\phi}$. Similarly, we continue to denote the meromorphic continuation of $\TwistedDualJSOfLocalFieldRepresentation{s}{W}{\phi}$ by the same symbol. Furthermore, denote $$ I_{\localFieldRepresentation, \fieldCharacter, \unitaryDetCharacter} = \Span_\cComplex \left\{ \TwistedJSOfLocalFieldRepresentation{s}{W}{\phi} \mid W \in \WhittakerModelOfLocalFieldRepresentation, \phi \in \Schwartz \left( \localField^m \right) \right\},$$ then there exists a unique element $p\left(Z\right) \in \cComplex \left[Z\right]$, such that $p \left(0\right) = 1$ and $I_{\localFieldRepresentation, \fieldCharacter, \unitaryDetCharacter} = \frac{1}{p \left( q^{-s} \right)} \cComplex \left[q^{-s}, q^{s}\right]$. $p \left(Z\right)$ does not depend on $\fieldCharacter$, and we denote $ \TwistedExteriorSquareLFunction{s}{\localFieldRepresentation}{\unitaryDetCharacter} = \frac{1}{p \left(q^{-s}\right)}$.
\end{thm}

\begin{prop}[{\cite[Proposition 3.2]{ye2018exterior}}]\label{prop:js-dual-formulas}
	\begin{enumerate}
		\item For $n = 2m$,
		\begin{align*}
			\TwistedDualJSOfLocalFieldRepresentation{s}{W}{\phi} = \int_{\lquot{\UnipotentSubgroup_m}{\GL{m}{\localField}}} \int_{\lquot{\UpperTriangularAdditive_m}{\SquareMat{m}{\localField}}}  & W \left( \evenPermutationMatrix \ShalikaUnipotentElement{X} \ShalikaDiagonalElement{g} \right) \fieldCharacter \left(- \trace X \right) \\
			& \cdot \abs{\det g}^{s-1} \unitaryDetCharacter \left( \det g \right) \FourierTransformWithRespectToCharacter{\phi}{\fieldCharacter} \left( \firstrowvector \inverseTranspose{g} \right) dX \multiplicativeMeasure{g},
		\end{align*}
		where $\firstrowvector = \begin{pmatrix}
		1 & 0 & \dots & 0
		\end{pmatrix}$.
		\item  For $n = 2m + 1$, \begin{align*}
			\TwistedDualJSOfLocalFieldRepresentation{s}{W}{\phi} =
			&\int_{\lquot{\UnipotentSubgroup_m}{\GL{m}{\localField}}} \int_{\Mat{1}{m}{\localField}} \int_{\lquot{\UpperTriangularAdditive_m}{\SquareMat{m}{\localField}}} \\
			& \cdot W \left( \antiDiagTwo{1}{\IdentityMatrix{2m}} \oddPermutationMatrix \ShalikaUnipotentElementOdd{X} \ShalikaDiagonalElementOdd{g} \ShalikaUpperRightUnipotentElementOdd{-\transpose{Z}} \right)  \\
			&\cdot \fieldCharacter \left(- \trace X \right) \abs{\det g}^{s} \unitaryDetCharacter \left( \det g \right) \FourierTransformWithRespectToCharacter{\phi}{\fieldCharacter} \left( Z \right) dX dZ \multiplicativeMeasure{g},
		\end{align*}
	\end{enumerate}
\end{prop}

\begin{thm}[{\cite[Theorem 4.1]{matringe2014linear} \cite[Theorem 3.1]{CogdellMatringe15}}]\label{thm:functional-equation-gamma-factor}
	There exists a non-zero element $\twistedGammaFactorOfLocalField{s} \in \cComplex \left(q^{-s}\right)$, such that for every $W \in \WhittakerModelOfLocalFieldRepresentation$, $\phi \in \Schwartz \left( \localField^m \right)$ we have $$ \TwistedDualJSOfLocalFieldRepresentation{s}{W}{\phi} = \twistedGammaFactorOfLocalField{s} \TwistedJSOfLocalFieldRepresentation{s}{W}{\phi}.$$
	Furthermore, $$ \twistedGammaFactorOfLocalField{s} = \TwistedEpsilonFactorOfLocalField{s} \frac{\TwistedExteriorSquareLFunction{1-s}{\Contragradient{\localFieldRepresentation}}{\unitaryDetCharacter^{-1}}}{\TwistedExteriorSquareLFunction{s}{\localFieldRepresentation}{\unitaryDetCharacter}},$$
	where $\TwistedEpsilonFactorOfLocalField{s} = c \cdot q^{-ks}$, for $k \in \zIntegers$ and $c \in \multiplicativegroup{\cComplex}$.
\end{thm}
The proof of the functional equation is very similar to the proofs of the referred theorems, and requires only slight modifications.

As before, we denote $\exteriorSquareLFunction{s}{\localFieldRepresentation} = \TwistedExteriorSquareLFunction{s}{\localFieldRepresentation}{1}$, $\gammaFactorOfLocalField{s} = \twistedGammaFactor{s}{\localFieldRepresentation}{1}$, and $\epsilonFactorOfLocalField{s} = \TwistedEpsilonFactor{s}{1}$.

\begin{thm}\label{thm:roots-of-l-function}Suppose that $ \representationDeclaration{\localFieldRepresentation}$ is an irreducible supercuspidal representation of $\GL{n}{\localField}$ 
	\begin{enumerate}
		\item If $n = 2m + 1$, then $\TwistedExteriorSquareLFunction{s}{\localFieldRepresentation}{\unitaryDetCharacter} = 1$.
		\item  If $n = 2m$, then $\TwistedExteriorSquareLFunction{s}{\localFieldRepresentation}{\unitaryDetCharacter} = \frac{1}{p\left(q^{-s}\right)}$, where $p\left(Z\right) \in \cComplex \left[Z\right]$ is a polynomial, such that $p\left(0\right) = 1$ and $p \left(Z\right) \mid 1 - \centralCharacter{\localFieldRepresentation} \left(\uniformizer\right) \unitaryDetCharacter\left(\uniformizer\right)^m Z^m$.
	\end{enumerate}
\end{thm}
The proof of this statement is very similar to the proof of {\cite[Theorem 3.6]{jo2018derivatives}}. Its proof uses a slight modification of \cite[Proposition 3.4]{jo2018derivatives} for the twisted Jacquet-Shalika integral. See also \cite[Section 8, Theorem 1]{jacquet11exterior}, \cite[Section 9, Theorem 2]{jacquet11exterior} for the analogous global statements.

\begin{lem}\label{lem:computation-of-exterior-square-factors}
	Let $n = 2m$. Suppose that $\twistedGammaFactorOfLocalField{s} = c \cdot q^{-ks} \frac{p_1 \left( q^{-s}\right)}{p_2 \left( q^{-\left(1 - s\right)}\right)}$, where $c \in \multiplicativegroup{\cComplex}$, $k \in \zIntegers$, $p_1, p_2 \in \cComplex \left[ Z \right]$, such that $p_1 \left(0\right) = p_2 \left(0\right) = 1$ and $p_1 \left(Z\right)$ and $p_2 \left(q^{-1} Z^{-1} \right)$ don't have any mutual roots. Then $\TwistedExteriorSquareLFunction{s}{\localFieldRepresentation}{\unitaryDetCharacter} = \frac{1}{p_1 \left( q^{-s} \right)}$, $\TwistedExteriorSquareLFunction{s}{\Contragradient{\localFieldRepresentation}}{\unitaryDetCharacter^{-1}} = \frac{1}{p_2 \left( q^{-s} \right)}$, $\TwistedEpsilonFactorOfLocalField{s} = c \cdot q^{-ks}$.
\end{lem}

\begin{proof}
	Write $\TwistedExteriorSquareLFunction{s}{\localFieldRepresentation}{\unitaryDetCharacter} = \frac{1}{p_\localFieldRepresentationIndex \left(q^{-s} \right)}$, $\TwistedExteriorSquareLFunction{s}{\Contragradient{\localFieldRepresentation}}{\unitaryDetCharacter^{-1}} = \frac{1}{p_\localFieldRepresentationContragradientIndex\left(q^{-s} \right)}$, and $\epsilonFactorOfLocalField{s} = c_\localFieldRepresentationIndex \cdot q^{-k_\localFieldRepresentationIndex \cdot  s}$, where $c_\localFieldRepresentationIndex \in \multiplicativegroup{\cComplex}$, $k_\localFieldRepresentationIndex \in \zIntegers$ and $p_\localFieldRepresentationIndex, p_\localFieldRepresentationContragradientIndex \in \cComplex \left[Z\right]$ satisfy $p_\localFieldRepresentationIndex \left(0\right) = p_\localFieldRepresentationContragradientIndex \left(0\right) = 1$.
	
	Then by \Cref{thm:functional-equation-gamma-factor} we have the equality $$\twistedGammaFactorOfLocalField{s} = c \cdot \left(q^{-s}\right)^k \frac{p_1\left( q^{-s}\right)}{p_2 \left( q^{-1} \left(q^{-s}\right)^{-1}\right)} = c_\localFieldRepresentationIndex \cdot \left(q^{-s}\right)^{k_\localFieldRepresentationIndex} \frac{p_\localFieldRepresentationIndex \left(q^{-s}\right)}{p_\localFieldRepresentationContragradientIndex \left(q^{-1} \left(q^{-s}\right)^{-1}\right)},$$ which implies that $$c Z^k p_1 \left(  Z\right) p_\localFieldRepresentationContragradientIndex \left(q^{-1} Z^{-1}\right) = c_\localFieldRepresentationIndex Z^{k_\localFieldRepresentationIndex} p_\localFieldRepresentationIndex \left(Z\right) p_2 \left(q^{-1} Z^{-1}\right),$$
	as elements of the polynomial ring $\cComplex\left[Z, Z^{-1}\right]$.
	
	By \Cref{thm:roots-of-l-function}, $p_\localFieldRepresentationIndex \left(Z\right) \mid 1 - \centralCharacter{\localFieldRepresentation}\left(\uniformizer\right) \unitaryDetCharacter \left( \uniformizer \right)^m Z^m$ and $p_\localFieldRepresentationContragradientIndex \left(Z\right) \mid 1 - \centralCharacter{\localFieldRepresentation}^{-1}\left(\uniformizer\right) \unitaryDetCharacter \left( \uniformizer \right)^{-m}  Z^m$. Therefore, we get that $p_\localFieldRepresentationIndex \left(Z\right)$ and $p_\localFieldRepresentationContragradientIndex \left(q^{-1} Z^{-1}\right)$ have no mutual roots. Note that they also don't have zero or infinity as a root. Therefore we conclude that every root of $p_\localFieldRepresentationIndex$ (including multiplicity) is a root of $p_1$, which implies that $p_1 \left(Z\right) = h_1 \left(Z\right) p_\localFieldRepresentationIndex \left(Z\right)$, where $h_1 \in \cComplex \left[Z\right]$, with $h_1 \left(0\right) = 1$. Similarly, we get that $p_2 \left(Z\right) = h_2 \left(Z\right) p_\localFieldRepresentationContragradientIndex \left(Z\right)$, where $h_2 \in \cComplex \left[Z\right]$, with $h_2 \left(0\right) = 1$. Hence we get that $c Z^k h_1 \left(Z\right) = c_\localFieldRepresentationIndex Z^{k_\localFieldRepresentationIndex} h_2 \left(q^{-1} Z^{-1}\right)$. Since $p_1 \left(Z\right)$ and $p_2 \left(q^{-1} Z^{-1}\right)$ don't have any mutual roots, and since both don't have zero or infinity as a root, we get that $h_1 \left(Z\right), h_2 \left( Z \right)$ are constants. Therefore $h_1 = h_2 = 1$, and the result follows.
\end{proof}

\subsection{Simple supercuspidal representations}\label{subsubsection:simple-supercuspidal}

Let $n$ be a positive integer.

Let $\multiplicativeCharacter : \multiplicativegroup{\localField} \rightarrow \multiplicativegroup{\cComplex} $ be a multiplicative character such that $\multiplicativeCharacter \restriction_{1 + \maximalideal} = 1$.

Let $\proUnipotentRadical_n = \quotientMap^{-1} \left( \residueFieldUnipotentSubgroup{n} \right)$ be the pro-unipotent radical of the standard Iwahori subgroup of $\GL{n}{\localField}$, where $\residueFieldUnipotentSubgroup{n}$ is the upper unipotent subgroup of $\GL{n}{\residueField}$. Denote $\centerTimesProUnipotent_n = \multiplicativegroup{\localField}\proUnipotentRadical_n$.

Let $t_0 \in \rquot{\multiplicativegroup{\integersring}}{1 + \maximalideal} \isomorphic \multiplicativegroup{\residueField}$. Let $t \in \multiplicativegroup{\integersring}$ be a lift of $t_0$, i.e. $\quotientMap \left(t\right) = t_0$. We define an affine generic character $\affineGenericCharacter : \centerTimesProUnipotent_n \rightarrow \multiplicativegroup{\cComplex}$ by $$\affineGenericCharacter \left( z k \right) = \multiplicativeCharacter \left(z\right) \fieldCharacter \left(\sum_{i = 1}^{n - 1}{a_i} + ta_n \right),$$ where $z \in \multiplicativegroup{\localField}$, and $$ k = \begin{pmatrix}
		x_1              & a_1  & \ast   & \cdots  & \ast    \\
		\ast             & x_2  & a_2    & \cdots  & \ast    \\
		\vdots           &      & \ddots & \ddots  & \vdots  \\
		\ast             & \ast & \cdots & x_{n-1} & a_{n-1} \\
		\uniformizer a_n & \ast & \cdots & \ast    & x_n
	\end{pmatrix} \in \proUnipotentRadical_n.$$
Note that $\affineGenericCharacter$ does not depend on the choice $t$, because the conductor of $\fieldCharacter$ is $\maximalideal$.

Let $\rootOfUniformizer \in \cComplex$ be an $n$th root of $\multiplicativeCharacter \left(\modifiedUniformizer\right)$.

Denote $\cyclicGroupGenerator{n} = \smallAntiDiagTwo{\IdentityMatrix{n-1}}{\modifiedUniformizer}$, $\modifiedCenterTimesProUnipotent_n = \grpGeneratedBy{\cyclicGroupGenerator{n}} \centerTimesProUnipotent_n$. We define a character $\affineGenericCharacter_{\rootOfUniformizer} : \modifiedCenterTimesProUnipotent_n \rightarrow \multiplicativegroup{\cComplex}$ by $\affineGenericCharacter_{\rootOfUniformizer} \left(\cyclicGroupGenerator{n}^j h \right) = \rootOfUniformizer^j \affineGenericCharacter \left(h\right)$, for $j \in \zIntegers$ and $h \in \centerTimesProUnipotent_n$.

\begin{thm}[{\cite[Section 4.3]{knightly2015simple}}]\label{thm:simple-supercuspidal}
	The representation $\simpleRep = \CompactInd{\modifiedCenterTimesProUnipotent_n}{\GL{n}{\localField}}{\affineGenericCharacter_{\rootOfUniformizer}}$ is an irreducible supercuspidal representation of $\GL{n}{\localField}$.
\end{thm}

A representation $\simpleRep$ such as in \cref{thm:simple-supercuspidal} is called a \textbf{simple supercuspidal representation}. We say that $\pi = \simpleRep$ is a simple supercuspidal representation with central character $\multiplicativeCharacter$, associated with the data $\left(t_0, \rootOfUniformizer\right)$. Simple supercuspidal representations were first constructed by Gross and Reeder in \cite{GrossReeder10} for groups that are simply connected, almost simple and split over the non-Archimedean field $\localField$.

By the proof of \cite[Corollary 5.3]{knightly2015simple}, there exist exactly $n \left(q - 1\right)$ equivalence classes of simple supercuspidal representations with a given central character, each of which corresponds to a pair $\left(t_0, \rootOfUniformizer\right)$.

By \cite[Section 3.3]{adrian2016some}, we have that if $W : \GL{n}{\localField} \rightarrow \cComplex$ is the function supported on $\UnipotentSubgroup_n \modifiedCenterTimesProUnipotent_n$ (where $\UnipotentSubgroup_n$ is the upper triangular unipotent subgroup of $\GL{n}{\localField}$), defined by \begin{align}\label{eq:simple-supercuspidal-whittaker-function}
	W \left(u h'\right) = \fieldCharacter \left(u\right) \affineGenericCharacter_{\rootOfUniformizer} \left(h'\right)& & u \in \UnipotentSubgroup_n, h' \in \modifiedCenterTimesProUnipotent_n,
\end{align} then $W \in \Whittaker \left(\simpleRep, \fieldCharacter\right)$ is a Whittaker function.

\section{Computation of the twisted exterior square factors}\label{sec:computation}

In this section, we compute the twisted exterior square factors of a simple supercuspidal representation. Throughout this section, let $t_0 \in \rquot{\multiplicativegroup{\integersring}}{1 + \maximalideal} \isomorphic \multiplicativegroup{\residueField}$, $t \in \multiplicativegroup{\integersring}$, $\multiplicativeCharacter : \multiplicativegroup{\localField} \rightarrow \multiplicativegroup{\cComplex}$, $\rootOfUniformizer \in \multiplicativegroup{\cComplex}$ be as in \Cref{subsubsection:simple-supercuspidal}. We denote $\localFieldRepresentation = \simpleRep$. Our goal is to compute the twisted exterior square factors of $\localFieldRepresentation$.

\subsection{Preliminary lemmas}

In order to compute the twisted exterior factors of $\localFieldRepresentation$, we will use the function $\localFieldRepresentation\left( \evenPermutationMatrix^{-1} \right) W$, where $W$ is the Whittaker function from \Cref{subsubsection:simple-supercuspidal}. Before beginning our computation, we need some lemmas regarding the support of the integrand of the twisted Jacquet-Shalika integral $\TwistedJSOfLocalFieldRepresentation{s}{\localFieldRepresentation \left( \evenPermutationMatrix^{-1} \right) W}{\phi}$.

Denote for $1 \le l \le m$, $\uniformizerDiagonalMatrix{l} = \smallDiagTwo{\IdentityMatrix{m - l}}{\modifiedUniformizer \IdentityMatrix{l}}$,
$\antidiagPermutationMatrix{l} = \smallAntiDiagTwo{\IdentityMatrix{m-l}}{\IdentityMatrix{l}}$, and denote by $\antidiagPermutation{l}$ the permutation defined by the columns of $\antidiagPermutationMatrix{l}$, i.e. $$\antidiagPermutationMatrix{l} =  \begin{pmatrix}
\standardColumnVector{\antidiagPermutation{l}\left(1\right)} & \dots & \standardColumnVector{\antidiagPermutation{l}\left(m\right)}
\end{pmatrix}.$$

\begin{lem} \label{lem:support-finite-field}
	Suppose that $g \in \GL{m}{\residueField}$, $X = \left(x_{ij}\right) \in \NilpotentLowerTriangular_m \left(\residueField\right)$ is a lower triangular nilpotent matrix, such that $$ \evenPermutationMatrix \ShalikaUnipotentElement{X} \ShalikaDiagonalElement{g} \evenPermutationMatrix^{-1} \in \residueFieldUnipotentSubgroup{2m} \antiDiagTwo{\IdentityMatrix{2m - 2l}}{\IdentityMatrix{2l}} \residueFieldUnipotentSubgroup{2m},$$
	for $1 \le l \le m$. Then
	\begin{enumerate}
		\item $g \in \residueFieldUnipotentSubgroupDoubleCoset{m}{l}$.
		\item If $g \in \antidiagPermutationMatrix{l} \residueFieldUnipotentSubgroup{m}$, then $x_{ij} = 0$ for every $j < i$ such that $\antidiagPermutation{l}^{-1}\left(j\right) < \antidiagPermutation{l}^{-1}\left(i\right)$, or equivalently $X \in \NilpotentLowerTriangular_m \left(\residueField\right) \cap \left( \antidiagPermutationMatrix{l} \NilpotentUpperTriangular_m \left(\residueField\right) \antidiagPermutationMatrix{l}^{-1} \right)$, where $\NilpotentUpperTriangular_m \left(\residueField\right)$ is the subgroup of $\SquareMat{m}{\residueField}$ consisting of upper triangular nilpotent matrices.
	\end{enumerate}
	Furthermore, for $g \in \antidiagPermutationMatrix{l} \residueFieldUnipotentSubgroup{m}$ and such $X$, $$\evenPermutationMatrix \ShalikaUnipotentElement{X} \ShalikaDiagonalElement{g} \evenPermutationMatrix^{-1} = \antiDiagTwo{\IdentityMatrix{2m - 2l}}{\IdentityMatrix{2l}} v,$$ where $v \in \residueFieldUnipotentSubgroup{2m}$ is an upper triangular unipotent matrix, having zeros right above its diagonal.
\end{lem}
\begin{proof}
	The lemma is proved in \cite[Lemma 2.28]{ye2018exterior} for the case that $g = wdu$, where $w$ is a permutation matrix, $d$ is a diagonal matrix and $u \in \residueFieldUnipotentSubgroup{m}$. Therefore we need only to show the first part for general $g$. By the Bruhat decomposition, we can write $g = u_1 wd u_2$, where $u_1, u_2 \in \residueFieldUnipotentSubgroup{m}$, $w$ is a permutation matrix, and $d$ is a diagonal matrix. Denote $g' = wdu_2$. We have $$\evenPermutationMatrix \ShalikaUnipotentElement{X} \ShalikaDiagonalElement{g} \evenPermutationMatrix^{-1} = \evenPermutationMatrix \ShalikaDiagonalElement{u_1} \evenPermutationMatrix^{-1} \evenPermutationMatrix \ShalikaUnipotentElement{u_1^{-1} X u_1} \ShalikaDiagonalElement{g'} \evenPermutationMatrix^{-1}.$$
	We have that $\evenPermutationMatrix \SmallShalikaDiagonalElement{u_1} \evenPermutationMatrix^{-1} \in \residueFieldUnipotentSubgroup{2m}$. Write $u_1^{-1}X u_1 = L + U$, where $L \in \NilpotentLowerTriangular_m\left(\residueField\right)$ is a lower triangular nilpotent matrix and $U \in \UpperTriangularAdditive_m\left(\residueField\right)$ is an upper triangular matrix. Then we have that $$\evenPermutationMatrix \ShalikaUnipotentElement{u_1^{-1}Xu_1} = \evenPermutationMatrix \ShalikaUnipotentElement{U} \evenPermutationMatrix^{-1} \evenPermutationMatrix \ShalikaUnipotentElement{L}.$$
	Since $\evenPermutationMatrix \SmallShalikaUnipotentElement{U} \evenPermutationMatrix^{-1} \in \residueFieldUnipotentSubgroup{2m}$, and since $\evenPermutationMatrix \SmallShalikaDiagonalElement{u_1} \evenPermutationMatrix^{-1} \in \residueFieldUnipotentSubgroup{2m}$, we get that $$\evenPermutationMatrix \ShalikaUnipotentElement{L} \ShalikaDiagonalElement{g'} \evenPermutationMatrix^{-1} \in \residueFieldUnipotentSubgroup{2m} \antiDiagTwo{\IdentityMatrix{2m - 2l}}{\IdentityMatrix{2l}} \residueFieldUnipotentSubgroup{2m}.$$
	Since $g' = wdu_2$, we get from \cite[Lemma 2.28]{ye2018exterior} that $wd = \antidiagPermutationMatrix{l}$, as required.
\end{proof}

\begin{lem}[{\cite[Lemma 2.29]{ye2018exterior}}] \label{lem:support-size-finite-field}
	\begin{enumerate}
		\item Let $d \in \GL{m}{\residueField}$ be a diagonal matrix. Then $\sizeof{\residueFieldUnipotentSubgroup{m} \antidiagPermutationMatrix{l} d \residueFieldUnipotentSubgroup{m}} = q^{\binom{m}{2} - \binom{l}{2} - \binom{m - l}{2}} \cdot \sizeof{\residueFieldUnipotentSubgroup{m}}$. Here $\binom{k}{2} = \frac{k \left(k  - 1\right)}{2}$, for any non-negative integer $k$.
		\item The set $$ \NilpotentLowerTriangular_m \left(\residueField\right) \cap \left(\antidiagPermutationMatrix{l} \NilpotentUpperTriangular_m \left(\residueField\right) \antidiagPermutationMatrix{l}^{-1}\right) = \left\{ \left(x_{ij}\right) \in \NilpotentLowerTriangular_m \left(\residueField\right) \mid x_{ij} = 0, \, \forall j < i \text{ s.t. } \antidiagPermutationMatrix{l}^{-1}\left(j\right) < \antidiagPermutationMatrix{l}^{-1}\left(i\right) \right\} $$ is of cardinality $q^{\binom{m}{2} - \binom{l}{2} - \binom{m-l}{2}}$. Here $\NilpotentUpperTriangular_m \left( \residueField \right)$ is the subgroup of $\SquareMat{m}{\residueField}$ consisting of upper triangular nilpotent matrices (i.e. upper triangular matrices with zeros on their diagonal).
	\end{enumerate}
\end{lem}
\begin{rem}
	In \cite[Lemma 2.29]{ye2018exterior} the lemma is stated only for $d$ a diagonal matrix of a certain form, but its proof only uses the fact that $d$ is a diagonal matrix.
\end{rem}

\begin{lem} \label{lem:abs-of-ai-elements}
	Suppose that $a = \diag\left(a_1, \dots, a_m\right)$ is an invertible diagonal matrix, and $X \in \NilpotentLowerTriangular_m \left(\localField\right)$ is a lower nilpotent matrix, such that
	\begin{equation}\label{eq:element-in-support-of-whittaker-function}
		\evenPermutationMatrix \ShalikaUnipotentElement{X} \ShalikaDiagonalElement{a} = \lambda \cdot u \cdot \cyclicGroupGenerator{2m}^r \cdot k,
	\end{equation}
	where $\lambda \in \multiplicativegroup{\localField}$, $u \in \UnipotentSubgroup_{2m}$, $1 \le r \le 2m$, $k \in \maximalCompactSubgroup_{2m}$. Then
	\begin{enumerate}
		\item $r = 2l$ is even, for some $1 \le l \le m$.
		\item $\abs{a_1} = \dots = \abs{a_{m-l}} = \abs{\lambda}$.
		\item $\abs{a_{m-l+1}} = \dots = \abs{a_{m}} = \abs{\lambda} \cdot \abs{\uniformizer}$.
		\item $\uniformizerDiagonalMatrix{l}^{-1} X \uniformizerDiagonalMatrix{l} \in \NilpotentLowerTriangular_m\left(\integersring\right)$.
	\end{enumerate}
\end{lem}

\begin{proof}
	\begin{enumerate}
		\item[1)] Taking the absolute value of the determinant of both sides of \cref{eq:element-in-support-of-whittaker-function}, we get $\abs{\det a}^2 = \abs{\lambda}^{2m} \cdot \abs{\det \cyclicGroupGenerator{2m}}^r$, and since $\abs{\det \cyclicGroupGenerator{2m}} = \abs{-\modifiedUniformizer} = q^{-1}$, we must have that $r$ is even. Thus $r = 2l$, for some $1 \le l \le m$. Then $$\cyclicGroupGenerator{2m}^{2l} = \antiDiagTwo{\IdentityMatrix{2m - 2l}}{\modifiedUniformizer \IdentityMatrix{2l}} = \diagTwo{\IdentityMatrix{2m-2l}}{\modifiedUniformizer \IdentityMatrix{2l}} \antiDiagTwo{\IdentityMatrix{2m - 2l}}{\IdentityMatrix{2l}}.$$
		\item[2\&3)] Denote $Z = a^{-1} X a$, $u_Z = \evenPermutationMatrix \SmallShalikaUnipotentElement{Z} \evenPermutationMatrix^{-1}$. Denote $$b = \evenPermutationMatrix \ShalikaDiagonalElement{a} \evenPermutationMatrix^{-1} = \diag \left(a_1, a_1, \dots, a_m, a_m\right).$$  Then $ b u_Z \evenPermutationMatrix = \lambda  u \cyclicGroupGenerator{2m}^{2l} k$.
		
		Let $u_Z = n_Z t_Z k_Z$ be an Iwasawa decomposition ($n_Z \in \UnipotentSubgroup_{2m}$, $t_Z \in \diagonalSubgroup_{2m}$, $k_Z \in \maximalCompactSubgroup_{2m}$). Then we have $\lambda^{-1} b t_Z = \left(b n_Z^{-1} b^{-1} u\right)    \cyclicGroupGenerator{2m}^{2l} \left(k \evenPermutationMatrix^{-1} k_Z^{-1}\right)$. Denote $u' = b n_Z^{-1} b^{-1} u \in \UnipotentSubgroup_{2m}$. Then we get $$\smallDiagTwo{\IdentityMatrix{2m-2l}}{\left(\modifiedUniformizer\right)^{-1} \IdentityMatrix{2l}} u'^{-1} \lambda^{-1} b t_Z \in \maximalCompactSubgroup_{2m}.$$ Writing $t_Z = \diag\left(t_1, \dots, t_{2m}\right)$, we get that $\abs{\lambda}^{-1} \abs{a_i} \abs{t_{2i}} = 1$ and $\abs{\lambda}^{-1} \abs{a_i} \abs{t_{2i -1}} = 1$, for every $1 \le i \le m-l$, and that $\abs{\lambda}^{-1} \abs{a_i} \abs{t_{2i}} = \abs{\uniformizer}$ and $\abs{\lambda}^{-1} \abs{a_i} \abs{t_{2i -1}} = \abs{\uniformizer}$, for every $m-l+1 \le i \le m$. By \cite[Section 5, Proposition 4]{jacquet11exterior}, $\abs{t_i} \ge 1$ for odd $i$ and $\abs{t_i} \le 1$ for even $i$. Thus we get that $\abs{t_i} = 1$ for every $i$. Hence, $\abs{a_1} = \dots = \abs{a_{m-l}} = \abs{\lambda}$ and $\abs{a_{m-l+1}} = \dots = \abs{a_{m}} = \abs{\lambda} \cdot \abs{\uniformizer}$.
		\item[4)]By \cite[Section 5, Proposition 5]{jacquet11exterior}, there exists $\alpha > 0$, such that if $Z = \left(z_{ij}\right)$, then $\max_{1 \le i,j \le m} \abs{z_{ij}}^{\alpha} \le \prod_{\substack{1 \le i \le 2m\\
				i \text{ odd}}} \abs{t_i}$. This implies that $Z = a^{-1} X a \in \SquareMat{m}{\integersring}$ since $\abs{t_i} = 1$. We have $a = \lambda \cdot \uniformizerDiagonalMatrix{l} \cdot k'$, where $k' \in \GL{m}{\integersring} \cap \diagonalSubgroup_m = {\left(\multiplicativegroup{\integersring}\right)}^m$, and this implies $\uniformizerDiagonalMatrix{l}^{-1} X \uniformizerDiagonalMatrix{l} \in \SquareMat{m}{\integersring}$, and therefore in $\NilpotentLowerTriangular_m\left(\localField\right) \cap \SquareMat{m}{\integersring} = \NilpotentLowerTriangular_m\left(\integersring\right)$.
	\end{enumerate}
\end{proof}

\begin{lem} \label{lem:support-of-integrand}
	Let $g = ak$, where $a = \diag \left( a_1, \dots, a_m \right)$ is an invertible matrix, $k \in \GL{m}{\integersring}$, $\multiplicativeMeasure{g} = \rightHaarMeasureModulus{\borelSubgroup_m} \left(a\right) \multiplicativeMeasure{k} \prod_{i = 1}^m{\multiplicativeMeasure{a_i}}$, and $X \in \NilpotentLowerTriangular_m\left(\localField\right)$ be a lower triangular nilpotent matrix. If $$ \evenPermutationMatrix \ShalikaUnipotentElement{X} \ShalikaDiagonalElement{g} \evenPermutationMatrix^{-1} \in \UnipotentSubgroup_{2m} \modifiedCenterTimesProUnipotent_{2m},$$
	then there exists $1 \le l \le m$, such that $\evenPermutationMatrix \SmallShalikaUnipotentElement{X} \SmallShalikaDiagonalElement{g} \evenPermutationMatrix^{-1} \in \multiplicativegroup{\localField} \UnipotentSubgroup_{2m} \cyclicGroupGenerator{2m}^{2l} \proUnipotentRadical_{2m}$. Moreover, if $\evenPermutationMatrix \SmallShalikaUnipotentElement{X} \SmallShalikaDiagonalElement{g} \evenPermutationMatrix^{-1} \in \lambda \UnipotentSubgroup_{2m} \cyclicGroupGenerator{2m}^{2l} \proUnipotentRadical_{2m}$ for $1 \le l \le m$ and $\lambda \in \multiplicativegroup{\localField}$, then
	\begin{enumerate}
		\item \label{item:diagonal-element-form} $a = \lambda \uniformizerDiagonalMatrix{l} \diag \left(u_1, \dots, u_m\right)$, where $u_1, \dots , u_{m} \in \multiplicativegroup{\integersring}$, $\rightHaarMeasureModulus{\borelSubgroup_m} \left(a\right) = \rightHaarMeasureModulus{\borelSubgroup_m} \left( \uniformizerDiagonalMatrix{l} \right) = q^{-l \left(m - l\right)}$.
		\item Let $k'' = \diag \left(u_1, \dots, u_m\right) k $. Then $\quotientMap \left( k'' \right) \in \residueFieldUnipotentSubgroupDoubleCoset{m}{l}$,  $\multiplicativeMeasure{k} = \multiplicativeMeasure{k''}$. 
		\item \label{item:X-form-lower-triangular-nilpotent} If $\quotientMap\left(k''\right) \in \antidiagPermutationMatrix{l} \residueFieldUnipotentSubgroup{m}$, then $X = \uniformizerDiagonalMatrix{l} Z \uniformizerDiagonalMatrix{l}^{-1}$ and $ dX = \rightHaarMeasureModulus{\borelSubgroup_m} \left( \uniformizerDiagonalMatrix{l} \right) dZ $, where $Z \in \NilpotentLowerTriangular_m\left(\integersring\right)$ satisfies $\quotientMap \left(Z\right) \in \NilpotentLowerTriangular_m \left(\residueField\right) \cap \left( \antidiagPermutationMatrix{l} \NilpotentUpperTriangular_m \left(\residueField\right) \antidiagPermutationMatrix{l}^{-1} \right)$. Moreover in this case, $$\evenPermutationMatrix \ShalikaUnipotentElement{X} \ShalikaDiagonalElement{g} \evenPermutationMatrix^{-1} = \lambda \cyclicGroupGenerator{2m}^{2l} v,$$ where $v \in \GL{2m}{\integersring}$ satisfies $\quotientMap \left(v\right) \in \residueFieldUnipotentSubgroup{2m}$, $\quotientMap \left(v\right)$ has zeros right above its diagonal, and $v$ has zero at its bottom left corner.
	\end{enumerate}
	\begin{proof}
		\begin{enumerate}
			\item Suppose that \begin{equation}\label{eq:element-in-support-of-whittaker-function-2}
			\evenPermutationMatrix \ShalikaUnipotentElement{X} \ShalikaDiagonalElement{a} \ShalikaDiagonalElement{k}  \evenPermutationMatrix^{-1} = \lambda u \cyclicGroupGenerator{2m}^{r} k',
			\end{equation} where $\lambda \in \multiplicativegroup{\localField}$, $u \in \UnipotentSubgroup_{2m}\left(\localField\right)$, $r \in \zIntegers$, $k' \in \proUnipotentRadical_{2m}$. Since $\cyclicGroupGenerator{2m}^{2m} = \modifiedUniformizer \IdentityMatrix{2m}$, we may assume (by modifying $\lambda$) that $1 \le r \le 2m$. By \cref{lem:abs-of-ai-elements}, we have that $r = 2l$, $X = \uniformizerDiagonalMatrix{l} Z \uniformizerDiagonalMatrix{l}^{-1}$, where $Z \in \NilpotentLowerTriangular_m\left(\integersring\right)$, and $a = \lambda \uniformizerDiagonalMatrix{l} \cdot \diag\left(u_1, \dots, u_m\right) $, for some $u_1,\dots,u_m \in \multiplicativegroup{\integersring}$.
			\item Let $k'' = \diag\left(u_1, \dots, u_m\right) \cdot k$ and $\uniformizerDiagonalMatrix{l}' = \smallDiagTwo{\IdentityMatrix{2m - 2l}}{\modifiedUniformizer \IdentityMatrix{2l}} $. Using these notations and part \ref{item:diagonal-element-form}, we have that \begin{equation} \label{eq:element-in-support-of-whittaker-function-in-terms-of-Z-and-k''}
				\evenPermutationMatrix \ShalikaUnipotentElement{X} \ShalikaDiagonalElement{g} \evenPermutationMatrix^{-1} = \lambda \evenPermutationMatrix \ShalikaDiagonalElement{\uniformizerDiagonalMatrix{l}} \ShalikaUnipotentElement{Z} \ShalikaDiagonalElement{k''} \evenPermutationMatrix^{-1}.
			\end{equation}
			Since $\uniformizerDiagonalMatrix{l}' = \evenPermutationMatrix \SmallShalikaDiagonalElement{\uniformizerDiagonalMatrix{l}} \evenPermutationMatrix^{-1}$, we get from \cref{eq:element-in-support-of-whittaker-function-in-terms-of-Z-and-k''}
			\begin{equation} \label{eq:element-in-support-of-whittaker-function-in-terms-of-Z-and-k''-2}
			\evenPermutationMatrix \ShalikaUnipotentElement{X} \ShalikaDiagonalElement{g} \evenPermutationMatrix^{-1} = \lambda \uniformizerDiagonalMatrix{l}' \evenPermutationMatrix \ShalikaUnipotentElement{Z} \ShalikaDiagonalElement{k''} \evenPermutationMatrix^{-1}.
			\end{equation}
			Recall that $r = 2l$. Writing $\cyclicGroupGenerator{2m}^{2l} =  \uniformizerDiagonalMatrix{l}'  \smallAntiDiagTwo{\IdentityMatrix{2m - 2l}}{\IdentityMatrix{2l}}$, we get by combining \cref{eq:element-in-support-of-whittaker-function-2} and \cref{eq:element-in-support-of-whittaker-function-in-terms-of-Z-and-k''-2} that \begin{equation}\label{eq:whittaker-support-elements-are-integral}
			\permutationMatrix{2m} \ShalikaUnipotentElement{Z} \ShalikaDiagonalElement{k''} \evenPermutationMatrix^{-1} = \uniformizerDiagonalMatrix{l}'^{-1} u \cyclicGroupGenerator{2m}^{2l} k' = \left(\uniformizerDiagonalMatrix{l}'^{-1} u \uniformizerDiagonalMatrix{l}'\right)
			\antiDiagTwo{\IdentityMatrix{2m - 2l}}{\IdentityMatrix{2l}} k',
			\end{equation}
			which implies that $\uniformizerDiagonalMatrix{l}'^{-1} u \uniformizerDiagonalMatrix{l}' \in \UnipotentSubgroup_{2m}\left(\integersring\right)$, as $\smallAntiDiagTwo{\IdentityMatrix{2m - 2l}}{\IdentityMatrix{2l}}, k' \in \maximalCompactSubgroup_{2m}$, and the left hand side of \cref{eq:whittaker-support-elements-are-integral} is in $\maximalCompactSubgroup_{2m}$. Since $\uniformizerDiagonalMatrix{l}'^{-1} u \uniformizerDiagonalMatrix{l}' \in \UnipotentSubgroup_{2m}\left(\integersring\right) \subseteq \proUnipotentRadical_{2m}$ and $k' \in \proUnipotentRadical_{2m}$, we get from \cref{eq:whittaker-support-elements-are-integral} that $$ \quotientMap \left( \permutationMatrix{2m} \ShalikaUnipotentElement{Z} \ShalikaDiagonalElement{k''} \evenPermutationMatrix^{-1} \right) \in  \residueFieldUnipotentSubgroup{2m} \antiDiagTwo{\IdentityMatrix{2m - 2l}}{\IdentityMatrix{2l}} \residueFieldUnipotentSubgroup{2m}.$$
			Since $Z \in \NilpotentLowerTriangular_m\left(\integersring\right)$, $\quotientMap\left(Z\right) \in \NilpotentLowerTriangular_m\left(\residueField\right)$, and by applying \Cref{lem:support-finite-field}, we have that $\quotientMap \left(k''\right) \in \residueFieldUnipotentSubgroupDoubleCoset{m}{l}$.
			
			\item Assume that $\quotientMap\left(k''\right) \in \antidiagPermutationMatrix{l} \residueFieldUnipotentSubgroup{m}$, then by \Cref{lem:support-finite-field} we have $\quotientMap \left( Z \right) \in \NilpotentLowerTriangular_m \left(\residueField\right) \cap \left(\antidiagPermutationMatrix{l} \NilpotentUpperTriangular_m \left(\residueField\right) \antidiagPermutationMatrix{l}^{-1} \right)$, and $$ \quotientMap \left( \permutationMatrix{2m} \ShalikaUnipotentElement{Z} \ShalikaDiagonalElement{k''} \evenPermutationMatrix^{-1} \right) = \antiDiagTwo{\IdentityMatrix{2m - 2l}}{\IdentityMatrix{2l}} v',$$ where $v' \in \residueFieldUnipotentSubgroup{2m}$ is an upper triangular matrix, having zeros right above its diagonal. Therefore \begin{equation}\label{eq:lift-of-element-in-the-support-of-residue-field-whittaker-function}
				\permutationMatrix{2m} \ShalikaUnipotentElement{Z} \ShalikaDiagonalElement{k''} \evenPermutationMatrix^{-1} = \antiDiagTwo{\IdentityMatrix{2m - 2l}}{\IdentityMatrix{2l}} v,
			\end{equation} where $v \in \GL{2m}{\integersring}$ satisfies $\quotientMap \left(v\right) = v'$. Combining \cref{eq:lift-of-element-in-the-support-of-residue-field-whittaker-function}, \cref{eq:element-in-support-of-whittaker-function-in-terms-of-Z-and-k''-2} and the fact that $\cyclicGroupGenerator{2m}^{2l} = \uniformizerDiagonalMatrix{l}' \smallAntiDiagTwo{\IdentityMatrix{2m - 2l}}{\IdentityMatrix{2l}}$, we get  $$ \evenPermutationMatrix \ShalikaUnipotentElement{X} \ShalikaDiagonalElement{g} \evenPermutationMatrix^{-1} = \lambda \uniformizerDiagonalMatrix{l}' \antiDiagTwo{\IdentityMatrix{2m - 2l}}{\IdentityMatrix{2l}} v  = \lambda \cyclicGroupGenerator{2m}^{2l} v.$$
			Finally, suppose that $l < m$. Note that a non-zero scalar multiple of the last row of $v$ appears as the $2m - 2l$ row of $\evenPermutationMatrix \SmallShalikaUnipotentElement{X} \SmallShalikaDiagonalElement{g} \evenPermutationMatrix^{-1}$. The $\left(2m - 2l , 1\right)$ coordinate of $\evenPermutationMatrix \SmallShalikaUnipotentElement{X} \SmallShalikaDiagonalElement{g} \evenPermutationMatrix^{-1}$ is the $\left(2m - l, 1\right)$ coordinate of $\SmallShalikaUnipotentElement{X} \SmallShalikaDiagonalElement{g}$, and this is zero, as $2m - l > m$.
			If $l = m$, then $\cyclicGroupGenerator{2m}^{2m} = \modifiedUniformizer \IdentityMatrix{2m}$, and therefore the last row of $v$ is a scalar multiple of the last row of $\evenPermutationMatrix \SmallShalikaUnipotentElement{X} \SmallShalikaDiagonalElement{g} \evenPermutationMatrix^{-1}$, which has zero as its first coordinate.
		\end{enumerate}
	\end{proof}
\end{lem}

\subsection{The even case} \label{subsection:even-computation}

In this section, we compute the twisted exterior square factors for the even case $n = 2m$.

\begin{thm}\label{thm:exterior-square-gamma-factor-for-even-case} Let $\localFieldRepresentation$ be a simple supercuspidal representation of $\GL{2m}{\localField}$ with central character $\multiplicativeCharacter$, associated with the data $\left(t_0, \rootOfUniformizer\right)$. 
	Let $\unitaryDetCharacter : \multiplicativegroup{\localField} \rightarrow \multiplicativegroup{\cComplex}$ be a unitary tamely ramified character, i.e. $\unitaryDetCharacter \restriction_{1 + \maximalideal} = 1$. Denote $\uniformizerValueForTwistedGamma = \rootOfUniformizer^2 \cdot \unitaryDetCharacter\left( \left(-1\right)^{m-1} \modifiedUniformizer \right)$. Let $\twistedMultiplicativeExteriorCharacter : \multiplicativegroup{\localField} \rightarrow \multiplicativegroup{\cComplex}$ be the character defined by $\twistedMultiplicativeExteriorCharacter \left( \uniformizer^j u \right) = \uniformizerValueForTwistedGamma^{j} \multiplicativeCharacter \left( u \right) \unitaryDetCharacter \left(u\right)^m$, for $j \in \zIntegers$, $u \in \multiplicativegroup{\integersring}$. Then $$\twistedGammaFactorOfLocalField{s} = \left( \uniformizerValueForTwistedGamma q^{-\left(s - \frac{1}{2}\right)} \right)^{m - 1} \gammaFactorOfCharacter{s}{ \twistedMultiplicativeExteriorCharacter}.$$ Explicitly,
	\begin{enumerate}
		\item If $\left(\multiplicativeCharacter \cdot \unitaryDetCharacter^m\right) \restriction_{\multiplicativegroup{\integersring}} \ne 1$, then $$ \twistedGammaFactorOfLocalField{s} = \left(\uniformizerValueForTwistedGamma q^{- \left(s - \frac{1}{2}\right)} \right)^{m-1}  \frac{1}{\sqrt{q}} \sum_{\lambda \in \multiplicativegroup{\residueField}}{\fieldCharacter \left(\lambda\right) \multiplicativeCharacter \left(\lambda^{-1}\right) \unitaryDetCharacter \left(\lambda^{-m}\right)}.$$ In this case $\TwistedExteriorSquareLFunction{s}{\localFieldRepresentation}{\unitaryDetCharacter} = 1$, $\TwistedEpsilonFactorOfLocalField{s} = \twistedGammaFactorOfLocalField{s}$.
		\item If $\left(\multiplicativeCharacter \cdot \unitaryDetCharacter^m \right)\restriction_{\multiplicativegroup{\integersring}} = 1$, then $$\twistedGammaFactorOfLocalField{s} = \left(\uniformizerValueForTwistedGamma q^{-\left(s - \frac{1}{2}\right)}\right)^{m-2} \frac{1 - \uniformizerValueForTwistedGamma  q^{-s}}{1 - \uniformizerValueForTwistedGamma^{-1} q^{- \left(1-s\right)}}.$$ In this case $\TwistedExteriorSquareLFunction{s}{\localFieldRepresentation}{\unitaryDetCharacter} = \frac{1}{1 - \uniformizerValueForTwistedGamma q^{-s}}$, $$\TwistedEpsilonFactorOfLocalField{s} = \uniformizerValueForTwistedGamma^{m - 2} q^{-\left(m - 2\right)\left(s - \frac{1}{2}\right)}.$$
	\end{enumerate}
\end{thm}
\begin{proof}
	We will compute the twisted exterior square gamma factor by computing the twisted Jacquet-Shalika integrals $\TwistedJSOfLocalFieldRepresentation{s}{\localFieldRepresentation \left(\evenPermutationMatrix^{-1}\right)W}{\phi}$ and $\TwistedDualJSOfLocalFieldRepresentation{s}{\localFieldRepresentation \left(\evenPermutationMatrix^{-1}\right)W}{\phi}$, where $W$ is the Whittaker function introduced in \Cref{subsubsection:simple-supercuspidal}, and $\phi : \localField^m \rightarrow \cComplex$ is the function defined by \begin{equation}\label{eqn:definition_of_phi}
    \phi \left(x\right) = \begin{cases}
	\fieldCharacter \left(-\quotientMap\left(x_1\right) \right) & x = \left(x_1, \dots, x_m\right) \in \integersring^m, \\
	0                                                           & \text{otherwise}
	\end{cases}.
    \end{equation}
    Then $$\FourierTransformWithRespectToCharacter{\phi}{\fieldCharacter} \left(x\right) = \begin{cases}
	q^{\frac{m}{2}} \indicatorFunction{\firstrowvector}\left(\quotientMap\left(x\right)\right) & x \in \integersring^m, \\
	0                                                                                          & \text{otherwise}
	\end{cases},$$ where $\indicatorFunction{\firstrowvector}\left(x\right)$ is the indicator function of $\firstrowvector = \left(1,0,\dots,0\right) \in \residueField^m$. 
	
	By the Iwasawa decomposition, we have that $\TwistedJSOfLocalFieldRepresentation{s}{\localFieldRepresentation \left(\evenPermutationMatrix^{-1}\right)W}{\phi}$ is given by
	
	\begin{equation}\label{eq:iwasawa-decomposition-for-jacquet-shalika-integral-even}
		\begin{split}
		J = \int& W \left( \evenPermutationMatrix \ShalikaUnipotentElement{X} \ShalikaDiagonalElement{ak} \evenPermutationMatrix^{-1} \right) \abs{\det a}^s \unitaryDetCharacter\left(\det \left(ak\right)\right)\\
		&\cdot \phi\left( \lastrowvector ak \right) \fieldCharacter\left(- \trace X \right) \rightHaarMeasureModulus{\borelSubgroup_m}\left(a\right) dX \multiplicativeMeasure{a} \multiplicativeMeasure{k},
		\end{split}
	\end{equation}
	where $W$ is the Whittaker function defined in \Cref{subsubsection:simple-supercuspidal}, given by \cref{eq:simple-supercuspidal-whittaker-function}, $X$ is integrated over $\lquot{\UpperTriangularAdditive_m}{\SquareMat{m}{\localField}}$, $k$ is integrated over $\maximalCompactSubgroup_m = \GL{m}{\integersring}$ and $a$ is integrated over the diagonal subgroup $\diagonalSubgroup_m$ of $\GL{m}{\localField}$.
	
	Denote by $\weylGroup{m}$ the group of $m \times m$ permutation matrices. By the Bruhat decomposition for $\GL{m}{\residueField}$, we have the disjoint union $$\GL{m}{\residueField} = \bigsqcup_{\substack{	w \in \weylGroup{m}\\
			d_0 \in \diagonalSubgroup_m\left(\residueField\right)}}{ \residueFieldUnipotentSubgroup{m} wd_0  \residueFieldUnipotentSubgroup{m} }.$$
	We decompose each of the double cosets of the disjoint union into a disjoint union of left cosets: given $w \in \weylGroup{m}$, $d_0 \in \diagonalSubgroup_m\left(\residueField\right)$, we can write
	$$\residueFieldUnipotentSubgroup{m} w d_0  \residueFieldUnipotentSubgroup{m} = \bigsqcup_{u_0 \in C_{wd_0}}{u_0 w d_0 \residueFieldUnipotentSubgroup{m}},$$
	where $C_{wd_0} \subseteq \residueFieldUnipotentSubgroup{m}$ is a subset of $\residueFieldUnipotentSubgroup{m}$ such that the map \begin{align*}
		C_{wd_0}  &\rightarrow \left\{ u_0 w d_0 \residueFieldUnipotentSubgroup{m} \mid u_0 \in \residueFieldUnipotentSubgroup{m} \right\},\\
		u_0 & \mapsto u_0 w d_0 \residueFieldUnipotentSubgroup{m}
	\end{align*} is a bijection. We may assume without loss of generality that $\IdentityMatrix{m} \in C_{w d_0}$.
	We have that $\sizeof{C_{w d_0}} = \frac{\sizeof{\residueFieldUnipotentSubgroup{m} w d_0 \residueFieldUnipotentSubgroup{m}}}{\sizeof{\residueFieldUnipotentSubgroup{m}}}$.
	
	We obtain the following decomposition:
	$$ \GL{m}{\residueField} = \bigsqcup_{\substack{w \in \weylGroup{m}\\
	d_0 \in \diagonalSubgroup_m\left(\residueField\right)}} \bigsqcup_{u_0 \in C_{wd_0}}{u_0 w d_0 \residueFieldUnipotentSubgroup{m}}.$$
	Since $\quotientMap^{-1}\left( \GL{m}{\residueField} \right) = \GL{m}{\integersring}$, we can lift the above decomposition to $$ \GL{m}{\integersring} = \bigsqcup_{\substack{w \in \weylGroup{m}\\
			d \in D_m}} \bigsqcup_{u \in C_{wd}}{uwd \quotientMap^{-1}\left( \residueFieldUnipotentSubgroup{m} \right) },$$
	where $D_m \subseteq \diagonalSubgroup_m \cap \GL{m}{\integersring} = \left(\multiplicativegroup{\integersring}\right)^m$ is a set of representatives for the inverse image $\quotientMap^{-1}\left( \diagonalSubgroup_m\left(\residueField\right) \right)$ (i.e. $D_m \subseteq \quotientMap^{-1}\left(\diagonalSubgroup_m\left(\residueField\right)\right)$ and $\quotientMap \restriction_{D_m} : D_m \rightarrow \diagonalSubgroup_m\left(\residueField\right)$ is a bijection), and for $d \in D_m$ with $\quotientMap\left(d\right) = d_0$, $C_{wd} \subseteq \UnipotentSubgroup_m\left(\integersring\right)$ is a set of representatives for the inverse image $\quotientMap^{-1}\left( C_{w d_0} \right)$ (i.e. $C_{wd} \subseteq \quotientMap^{-1}\left(C_{w d_0}\right)$ and $\quotientMap\restriction_{C_{wd}} : C_{wd} \rightarrow C_{w d_0}$ is a bijection). Without loss of generality, we may assume that the identity matrix belongs to $D_m$, and also  belongs to $C_{wd}$, for every $w \in \weylGroup{m}$ and $d \in D_m$.
	
	Using this decomposition for $\maximalCompactSubgroup_m$ in \cref{eq:iwasawa-decomposition-for-jacquet-shalika-integral-even}, we decompose the integral $J$ into a sum of integrals $$J = \sum_{\substack{w \in \weylGroup{m}\\
			d \in D_m}}\sum_{u \in C_{wd}}{J_{wd,u}},$$ where \begin{align*}
		J_{wd,u} = \int& W \left( \evenPermutationMatrix \ShalikaUnipotentElement{X} \ShalikaDiagonalElement{a uwd k} \evenPermutationMatrix^{-1} \right) \abs{\det a}^s \unitaryDetCharacter\left(\det \left(a uwd k\right)\right)\\
		&\cdot \phi\left( \lastrowvector a uwd k \right) \fieldCharacter\left(- \trace X \right) \rightHaarMeasureModulus{\borelSubgroup_m}\left(a\right) dX \multiplicativeMeasure{a} \multiplicativeMeasure{k},
	\end{align*}
	where $X$ is integrated over $\lquot{\UpperTriangularAdditive_m}{\SquareMat{m}{\localField}}$, $k$ is integrated over $\proUnipotentRadical_m = \quotientMap^{-1}\left(\residueFieldUnipotentSubgroup{m}\right)$, and $a$ is integrated over $\diagonalSubgroup_m$. Writing $au = aua^{-1} \cdot a$, we have that $a u a^{-1} \in \UnipotentSubgroup_m$, and since the Jacquet-Shalika integrand is invariant under $\lquot{\UnipotentSubgroup_m}{\GL{m}{\localField}}$, we have $J_{wd,\IdentityMatrix{m}} = J_{wd,u}$ for any $u \in C_{wd}$. Denote $J_{wd} = J_{wd,\IdentityMatrix{m}}$, then we have
	$$J = \sum_{\substack{w \in \weylGroup{m}\\
			d \in D_m}} \sizeof{C_{wd}} J_{wd} = \sum_{\substack{w \in \weylGroup{m}\\
			d \in D_m}} \frac{\sizeof{\residueFieldUnipotentSubgroup{m} w \quotientMap\left(d\right) \residueFieldUnipotentSubgroup{m}}}{\sizeof{\residueFieldUnipotentSubgroup{m}}} J_{wd}.$$
	
	Using the isomorphism $\lquot{\UpperTriangularAdditive_m}{\SquareMat{m}{\localField}} \cong \NilpotentLowerTriangular_m$, we can write 
	\begin{equation}\label{eq:expresion-for-j-w-d-even-case}
		\begin{split}
		J_{wd} = \int& W \left( \evenPermutationMatrix \ShalikaUnipotentElement{X} \ShalikaDiagonalElement{a wd k} \evenPermutationMatrix^{-1} \right) \abs{\det a}^s \unitaryDetCharacter\left(\det \left(a wd k\right)\right)\\
		&\cdot \phi\left( \lastrowvector a wd k \right) \rightHaarMeasureModulus{\borelSubgroup_m}\left(a\right) dX \multiplicativeMeasure{a} \multiplicativeMeasure{k},
		\end{split}
	\end{equation}
	where the integration is the same as in $J_{wd,u}$, except that this time $X$ is integrated on $\NilpotentLowerTriangular_m$.
	
	By \Cref{lem:support-of-integrand}, $J_{wd} = 0$ unless $w = \antidiagPermutationMatrix{l}$ for some $1 \le l \le m$. In this case by \Cref{lem:support-of-integrand}, we have that the integrand of $J_{\antidiagPermutationMatrix{l} d}$ is supported on $\uniformizerDiagonalMatrix{l} \cdot \multiplicativegroup{\localField} \cdot \diagonalSubgroup_m\left(\integersring\right)$, where $\diagonalSubgroup_m\left(\integersring\right) = \diagonalSubgroup_m \cap \GL{m}{\integersring} \isomorphic \left(\multiplicativegroup{\integersring}\right)^m$. We translate $a$ by $d_l$ and write down the expression for the Haar measure for the subgroup $\multiplicativegroup{\localField} \cdot \diagonalSubgroup_m\left(\integersring\right)$: \begin{equation}\label{eq:parameterization-of-a-even-case}
		\begin{aligned}
		a &= \uniformizerDiagonalMatrix{l} \lambda \diag\left(u_1,\dots,u_m\right),  \text{ where } \lambda \in \multiplicativegroup{\localField},u_1,\dots,u_m \in \multiplicativegroup{\integersring}. \\
		\multiplicativeMeasure{a} &= \multiplicativeMeasure{\lambda} \prod_{i=1}^{m}{\multiplicativeMeasure{u_i}},\\
		\rightHaarMeasureModulus{\borelSubgroup_m}\left(a\right) &= \rightHaarMeasureModulus{\borelSubgroup_m}\left(\uniformizerDiagonalMatrix{l}\right) = q^{-l \left(m-l\right)}.
		\end{aligned}
	\end{equation} Denote \begin{equation}\label{eq:definition-of-k''-even-case}
		k'' = \diag\left(u_1,\dots,u_m\right) \weylElement{l} d k.
	\end{equation} By \Cref{lem:support-of-integrand}, $k''$ satisfies $\quotientMap\left(k''\right) \in \residueFieldUnipotentSubgroup{m} \weylElement{l} \residueFieldUnipotentSubgroup{m}$. Since $\quotientMap\left(k\right) \in \residueFieldUnipotentSubgroup{m}$, then by the Bruhat decomposition of $\quotientMap\left(k''\right)$, we must have $$\quotientMap\left(\diag\left(u_1,\dots,u_m\right)\right) \antidiagPermutationMatrix{l} \quotientMap\left(d\right) = \antidiagPermutationMatrix{l}.$$
	Therefore, we have \begin{equation}\label{eq:paramterization-of-u-even-case}
		\begin{aligned}
		\diag\left(u_1,\dots,u_m\right) &= \diag\left(u_1',\dots,u_m'\right) \antidiagPermutationMatrix{l} d^{-1} \antidiagPermutationMatrix{l}^{-1},
		\end{aligned}
	\end{equation}
    where $u_1',\dots,u_m' \in 1 + \maximalideal$ and $\prod_{i=1}^m \multiplicativeMeasure{u_i} = \prod_{i=1}^m \multiplicativeMeasure{u'_i}$.
	Denote $g = a \antidiagPermutationMatrix{l} dk$. By \cref{eq:parameterization-of-a-even-case} and \cref{eq:paramterization-of-u-even-case}, we have \begin{equation}\label{eq:decomposition-of-g-to-diagonal-matrices-even-case}
		 g = a \antidiagPermutationMatrix{l} dk = \lambda d_l \diag\left(u_1',\dots,u_m'\right) \antidiagPermutationMatrix{l} k.
	\end{equation}
	By \cref{eq:definition-of-k''-even-case} and \cref{eq:paramterization-of-u-even-case}, we have $k'' = \diag\left(u_1',\dots,u_m'\right) \antidiagPermutationMatrix{l} k$, and therefore $\quotientMap\left(k''\right) \in \antidiagPermutationMatrix{l} \residueFieldUnipotentSubgroup{m}$. Hence, by part \ref{item:X-form-lower-triangular-nilpotent} of \Cref{lem:support-of-integrand}, \begin{equation}\label{eq:parameterization-of-X-even-case}
		\begin{aligned}
		X & = \uniformizerDiagonalMatrix{l} Z \uniformizerDiagonalMatrix{l}^{-1}, \text{ where } Z \in \quotientMap^{-1}\left( \NilpotentLowerTriangular_m\left(\residueField\right) \cap \left(\antidiagPermutationMatrix{l} \NilpotentUpperTriangular_m\left(\residueField\right) \antidiagPermutationMatrix{l}^{-1} \right) \right),\\
		dX &= \rightHaarMeasureModulus{\borelSubgroup_m}\left(\uniformizerDiagonalMatrix{l}\right)dZ = q^{-l \left(m-l\right)}dZ.
		\end{aligned}
	\end{equation}
	Moreover, we have that \begin{equation}\label{eq:whittaker-function-argument-as-element-in-whittaker-support}
		\permutationMatrix{2m} \ShalikaUnipotentElement{X} \ShalikaDiagonalElement{g}  \permutationMatrix{2m}^{-1} = \lambda \cyclicGroupGenerator{2m}^{2l} v,
	\end{equation} where $v \in \GL{2m}{\integersring}$ satisfies $\quotientMap\left(v\right) \in \residueFieldUnipotentSubgroup{2m}$, $\quotientMap\left(v\right)$ has zeros right above its diagonal, and $v$ has zero at its bottom left corner. Therefore by \cref{eq:simple-supercuspidal-whittaker-function} and \cref{eq:whittaker-function-argument-as-element-in-whittaker-support},  in this domain \begin{equation}\label{eq:evaluation-of-whittaker-function-even-case}
		W \left( \permutationMatrix{2m} \ShalikaUnipotentElement{X} \ShalikaDiagonalElement{g}  \permutationMatrix{2m}^{-1}\right) =  \rootOfUniformizer^{2l} \multiplicativeCharacter\left(\lambda\right).
	\end{equation}
	
	We have by \cref{eq:decomposition-of-g-to-diagonal-matrices-even-case} \begin{equation}\label{eq:evaluation-of-determinant-even-case}
		\det\left(a \antidiagPermutationMatrix{l} dk\right) = \lambda^m  \left(\modifiedUniformizer\right)^l \left(\prod_{i=1}^{m}{u'_i}\right) \left(-1\right)^{l \left(m - l\right)} \det k,
	\end{equation} and
	\begin{equation}\label{eq:evaluation-of-abs-determinant-even-case}
	\abs{\det a}^s = \abs{\lambda}^{ms} q^{-ls}.
	\end{equation}
	Since $\det k \in 1 +\maximalideal$,  $\prod_{i=1}^m{u'_i} \in 1 + \maximalideal$, and since $\unitaryDetCharacter$ is a tamely ramified character, and since $\left(-1\right)^{l^2} = \left(-1\right)^l$, we get by \cref{eq:evaluation-of-determinant-even-case} \begin{equation}\label{eq:evaluation-of-unitary-det-character}
		\unitaryDetCharacter\left(\det\left(a \antidiagPermutationMatrix{l} dk\right)\right) = \unitaryDetCharacter\left(\lambda\right)^m \unitaryDetCharacter\left( \left(-1\right)^{m-1} \modifiedUniformizer\right)^l.
	\end{equation}
	
	We have by \cref{eq:decomposition-of-g-to-diagonal-matrices-even-case} that $\lastrowvector a \antidiagPermutationMatrix{l} d k = u'_m \lambda \modifiedUniformizer \rowvector{l} k$. Since $u'_m \in 1 + \maximalideal \subseteq \multiplicativegroup{\integersring}$, $x \in \localField^m$ satisfies $x \in \integersring^m$ if and only if $u'_m x \in \integersring^m$. In the case $x \in \integersring^m$, we have $x \equiv u'_m x \left(\modulooperator \maximalideal\right)$, which implies that if $x = \left(x_1, \dots, x_m\right)$ then $x_1 \equiv u'_m x_1 \left(\modulooperator \maximalideal \right)$. Since $\fieldCharacter$ has conductor $\maximalideal$, $\fieldCharacter\left( u'_m x_1 \right) = \fieldCharacter\left(x_1\right)$. Therefore, by the definition of $\phi$ in \cref{eqn:definition_of_phi}, $\phi\left(u'_m x\right) = \phi\left(x\right)$ for every $x \in \localField^m$. It follows that \begin{equation}\label{eq:evaluation-of-phi-even-case}
		\phi\left(\lastrowvector a \antidiagPermutationMatrix{l} d k\right) = \phi\left(u'_m \lambda \modifiedUniformizer \rowvector{l}  k\right) = \phi\left(\lambda \modifiedUniformizer \rowvector{l} k\right).
	\end{equation}
	
	Substituting in \cref{eq:expresion-for-j-w-d-even-case} the equalities \cref{eq:parameterization-of-a-even-case}, \cref{eq:paramterization-of-u-even-case},
	\cref{eq:parameterization-of-X-even-case},  \cref{eq:evaluation-of-whittaker-function-even-case}, \cref{eq:evaluation-of-abs-determinant-even-case}, \cref{eq:evaluation-of-unitary-det-character} and \cref{eq:evaluation-of-phi-even-case}, we get \begin{equation}\label{eq:evaluation-of-expressions-in-integrand-even-case}
		\begin{split}
		J_{\antidiagPermutationMatrix{l} d} = & \int  \left(\rootOfUniformizer^{2l} \multiplicativeCharacter\left(\lambda\right)\right) \left(\abs{\lambda}^{m s} q^{-ls}\right) \left(\unitaryDetCharacter\left(\lambda\right)^m \unitaryDetCharacter\left( \left(-1\right)^{m-1} \modifiedUniformizer \right)^l\right) \\
		& \cdot   \phi\left( \lambda \modifiedUniformizer \rowvector{l} k \right) q^{-l\left(m-l\right)}  \left(q^{-l\left(m-l\right)} dZ\right) \left(\multiplicativeMeasure{\lambda} \prod_{i=1}^m{\multiplicativeMeasure{u'_i}}\right) \multiplicativeMeasure{k},
		\end{split}
	\end{equation}
	where the integral is integrated over $\lambda \in \multiplicativegroup{\localField}$, $u'_1,\dots,u'_m \in 1 + \maximalideal$, $Z \in \quotientMap^{-1}\left( \NilpotentLowerTriangular_m\left(\residueField\right) \cap \left(\antidiagPermutationMatrix{l} \NilpotentUpperTriangular_m\left(\residueField\right) \antidiagPermutationMatrix{l}^{-1} \right) \right)$, $k \in \proUnipotentRadical_m$.
	
	Denote $\uniformizerValueForTwistedGamma = \rootOfUniformizer^2 \unitaryDetCharacter\left( \left(-1\right)^{m-1}  \modifiedUniformizer \right)$. We can now evaluate the integration over $Z$,$u_1',\dots,u_{m}'$ in \cref{eq:evaluation-of-expressions-in-integrand-even-case} and get \begin{equation}\label{eq:partial-evaluation-of-integral-even-case}
		\begin{split}
		J_{\antidiagPermutationMatrix{l} d} = & q^{-2l\left(m-l\right)} q^{-ls} \cdot \frac{\sizeof{\NilpotentLowerTriangular_m\left(\residueField\right) \cap \left(\antidiagPermutationMatrix{l} \NilpotentUpperTriangular_m\left(\residueField\right) \antidiagPermutationMatrix{l}^{-1} \right)}}{\sizeof{\NilpotentLowerTriangular_m\left(\residueField\right)}} \frac{1}{\sizeof{\multiplicativegroup{\residueField}}^{m}} {\uniformizerValueForTwistedGamma}^l\\
		& \cdot \int_{\proUnipotentRadical_m} \int_{\multiplicativegroup{\localField}} \multiplicativeCharacter\left(\lambda\right) \unitaryDetCharacter\left(\lambda\right)^m \abs{\lambda}^{m s} \phi\left( \lambda \modifiedUniformizer \rowvector{l} k \right) \multiplicativeMeasure{\lambda} \multiplicativeMeasure{k}.
		\end{split}
	\end{equation}
	Notice that \cref{eq:partial-evaluation-of-integral-even-case} implies that $J_{\antidiagPermutationMatrix{l} d}$ does not depend on $d \in D_m$, and we have $J_{\antidiagPermutationMatrix{l}} = J_{\antidiagPermutationMatrix{l}d}$ for every $d \in D_m$.
	Denote \begin{equation}\label{eq:definition-of-jl-even-case}
		J_l = \sum_{d \in D_m}\sizeof{C_{\antidiagPermutationMatrix{l} d}} J_{\antidiagPermutationMatrix{l}d} = \sum_{d \in D_m} \frac{\sizeof{\residueFieldUnipotentSubgroup{m} \antidiagPermutationMatrix{l} \quotientMap\left(d\right) \residueFieldUnipotentSubgroup{m}}}{\sizeof{\residueFieldUnipotentSubgroup{m}}} J_{\antidiagPermutationMatrix{l}}.
	\end{equation}
	By \Cref{lem:support-size-finite-field},  \begin{equation}\label{eq:sizes-of-residue-field-sets-even-case}
		\frac{\sizeof{\residueFieldUnipotentSubgroup{m} \antidiagPermutationMatrix{l} \quotientMap\left(d\right) \residueFieldUnipotentSubgroup{m}}}{\sizeof{\residueFieldUnipotentSubgroup{m}}} = \sizeof{\NilpotentLowerTriangular_m\left(\residueField\right) \cap \left(\antidiagPermutationMatrix{l} \NilpotentUpperTriangular_m\left(\residueField\right) \antidiagPermutationMatrix{l}^{-1} \right)} = q^{ \binom{m}{2} - \binom{l}{2} - \binom{m - l}{2} }.
	\end{equation} We also have $\sizeof{D_m} = \sizeof{\diagonalSubgroup_m\left(\residueField\right)} = \sizeof{\multiplicativegroup{\residueField}}^m$. Therefore by substituting into \cref{eq:definition-of-jl-even-case} the equalities  \cref{eq:partial-evaluation-of-integral-even-case} and \cref{eq:sizes-of-residue-field-sets-even-case}, we have
	\begin{equation}\label{eqn:partial-evaluation-of-J-l-even-case}
		J_{l} = q^{-\binom{m}{2}} q^{-l s} \uniformizerValueForTwistedGamma^l \cdot \int_{\proUnipotentRadical_m} \int_{\multiplicativegroup{\localField}} \multiplicativeCharacter\left(\lambda\right) \unitaryDetCharacter\left(\lambda\right)^m \abs{\lambda}^{m s} \phi\left( \lambda \modifiedUniformizer \rowvector{l} k \right) \multiplicativeMeasure{\lambda} \multiplicativeMeasure{k},
	\end{equation}
	where the expression $q^{-\binom{m}{2}}$ arises from the identity $- 2l\left(m - l\right) + 2\binom{m}{2} - 2\binom{l}{2} - 2\binom{m-l}{2} - \binom{m}{2} = -\binom{m}{2}$.
	
	Since $k \in \proUnipotentRadical_m$, we have that $\rowvector{l} k \in \integersring^m$ has $1$ as its $l$ coordinate modulo $\maximalideal$. Therefore if $\lambda \modifiedUniformizer \varepsilon_l k \in \integersring^m$, we must have $\abs{\lambda \modifiedUniformizer} \le 1$, i.e. $\lambda = \left(\modifiedUniformizer\right)^j \cdot u_0$, for some $u_0 \in \multiplicativegroup{\integersring}$ and $j \ge -1$. For a fixed $k \in \proUnipotentRadical_m$ we decompose
	\begin{equation}\label{eq:tate-integral-even-case}
		\begin{split}
		&\int_{\multiplicativegroup{\localField}} \multiplicativeCharacter\left(\lambda\right) \unitaryDetCharacter\left(\lambda\right)^m \abs{\lambda}^{m s} \phi\left( \lambda \modifiedUniformizer \rowvector{l} k \right) \multiplicativeMeasure{\lambda}
		\\
		& = \sum_{j = -1}^{\infty} \multiplicativeCharacter\left(\modifiedUniformizer\right)^j \unitaryDetCharacter\left(\modifiedUniformizer\right)^{j m} q^{-j m s} \cdot \int_{\multiplicativegroup{\integersring}} \multiplicativeCharacter\left(u_0\right) \unitaryDetCharacter\left(u_0\right)^m \phi\left(u_0 \left(\modifiedUniformizer\right)^{j+1} \rowvector{l} k \right) \multiplicativeMeasure{u_0} 
		\end{split}
	\end{equation}
	Since $\uniformizerValueForTwistedGamma^m = \rootOfUniformizer^{2m} \unitaryDetCharacter\left( \left(-1\right)^{m \left(m - 1\right)}  \right) \unitaryDetCharacter\left(\modifiedUniformizer\right)^m$, $m \left(m - 1\right)$ is even, and $\rootOfUniformizer^{2m} = \multiplicativeCharacter\left(\modifiedUniformizer\right)$, we have $\uniformizerValueForTwistedGamma^m = \multiplicativeCharacter\left( \modifiedUniformizer \right) \unitaryDetCharacter\left( \modifiedUniformizer \right)^m$. Therefore, we get from \cref{eq:tate-integral-even-case} that
    \begin{equation}\label{eq:tate-integral-2-even-case}
	\begin{split}
	&\int_{\multiplicativegroup{\localField}} \multiplicativeCharacter\left(\lambda\right) \unitaryDetCharacter\left(\lambda\right)^m \abs{\lambda}^{m s} \phi\left(  \lambda \modifiedUniformizer \rowvector{l} k \right) \multiplicativeMeasure{\lambda}
	\\
	& = \sum_{j = -1}^{\infty} \left(\uniformizerValueForTwistedGamma q^{-s}\right)^{jm} \cdot \int_{\multiplicativegroup{\integersring}} \multiplicativeCharacter\left(u_0\right) \unitaryDetCharacter\left(u_0\right)^m \phi\left(u_0 \left(\modifiedUniformizer\right)^{j+1} \rowvector{l}  k \right) \multiplicativeMeasure{u_0}.
	\end{split}
    \end{equation}
	If $l \ge 2$, then $\rowvector{l} k$ has $0$ as its first coordinate modulo $\maximalideal$, so for every $\displaystyle{j \ge -1}$, the first coordinate of $u_0 \left(\modifiedUniformizer\right)^{j+1} \rowvector{l} k$ is $0$ modulo $\maximalideal$. Thus $\phi\left(u_0 \left(\modifiedUniformizer\right)^{j+1} \rowvector{l} k \right) = 1$. We also have that $$\int_{\multiplicativegroup{\integersring}}{\multiplicativeCharacter\left(u_0\right)\unitaryDetCharacter\left(u_0\right)^m}\multiplicativeMeasure{u_0} = \begin{cases}
	1 & \left(\multiplicativeCharacter \cdot \unitaryDetCharacter^m \right)\restriction_{\multiplicativegroup{\integersring}} = 1\\
	0 & \text{otherwise}
	\end{cases}.$$
	Therefore, from \cref{eqn:partial-evaluation-of-J-l-even-case} and \cref{eq:tate-integral-2-even-case}, we get for $l \ge 2$ that, $$J_l = \begin{cases}
	\frac{q^{-\binom{m}{2}}}{\grpIndex{\GL{m}{\residueField}}{\residueFieldUnipotentSubgroup{m}}} \left( \uniformizerValueForTwistedGamma q^{-s} \right)^{-m} \left( \uniformizerValueForTwistedGamma q^{-s} \right)^{l} \frac{1 }{1 - \uniformizerValueForTwistedGamma^m q^{-ms}} & \left(\multiplicativeCharacter \cdot \unitaryDetCharacter^m \right) \restriction_{\multiplicativegroup{\integersring}} = 1\\
	0 & \text{otherwise}
	\end{cases}.$$
	
	If $l = 1$ and $j \ge 0$, we have that $\left(\modifiedUniformizer\right)^{j+1} \rowvector{l} u_0  k \equiv 0 \left(\modulooperator \maximalideal \right)$, and therefore we have again $\phi\left( u_0 \left(\modifiedUniformizer\right)^{j+1} \rowvector{l} k \right) = 1$ and $$\int_{\multiplicativegroup{\integersring}}{\multiplicativeCharacter\left(u_0\right)\unitaryDetCharacter\left(u_0\right)^m}\multiplicativeMeasure{u_0} = \begin{cases}
	1 & \left(\multiplicativeCharacter \cdot \unitaryDetCharacter^m \right)\restriction_{\multiplicativegroup{\integersring}} = 1\\
	0 & \text{otherwise}
	\end{cases}.$$ If $l = 1$ and $j = -1$, we have that $\firstrowvector k$ has $1$ as its first coordinate modulo $\maximalideal$, and therefore $\phi\left( u_0 \left(\modifiedUniformizer\right)^{j+1} \rowvector{l} k \right) = \fieldCharacter\left(-\quotientMap\left(u_0\right) \right)$, and we have $$\int_{\multiplicativegroup{\integersring}} \multiplicativeCharacter\left(u_0\right) \unitaryDetCharacter\left(u_0\right)^m \fieldCharacter\left(-u_0\right) \multiplicativeMeasure{u_0} = \frac{1}{\sizeof{\multiplicativegroup{\residueField}}} \sum_{\lambda \in \multiplicativegroup{\residueField}}{ \multiplicativeCharacter\left(\lambda\right) \unitaryDetCharacter\left(\lambda\right)^m \fieldCharacter\left(-\lambda\right)}.$$
	To summarize, we get: $$J_1 = \begin{cases}
	\frac{q^{-\binom{m}{2}}}{\grpIndex{\GL{m}{\residueField}}{\residueFieldUnipotentSubgroup{m}}} \left( \uniformizerValueForTwistedGamma q^{-s} \right)^{-m} \left(\frac{\left({\uniformizerValueForTwistedGamma q^{-s}}\right)^{m+1} }{1 - \uniformizerValueForTwistedGamma^m q^{-ms}} - \frac{\uniformizerValueForTwistedGamma q^{-s}}{q-1}\right) & \left(\multiplicativeCharacter \cdot \unitaryDetCharacter^m \right)\restriction_{\multiplicativegroup{\integersring}} = 1, \\
	\frac{ q^{-\binom{m}{2}}  }{\grpIndex{\GL{m}{\residueField}}{\residueFieldUnipotentSubgroup{m}}} \left( \uniformizerValueForTwistedGamma q^{-s} \right)^{-m} \left( \uniformizerValueForTwistedGamma q^{-s} \right) \frac{1}{\sizeof{\multiplicativegroup{\residueField}}} \sum_{\lambda \in \multiplicativegroup{\residueField}}{ \multiplicativeCharacter\left(\lambda\right) \unitaryDetCharacter\left(\lambda\right)^m \fieldCharacter\left(-\lambda\right)} & \text{otherwise}
	\end{cases}.$$
	Summing all the $J_l$ up, we get \begin{equation}\label{eq:computation-of-J-even-case}
		J = \sum_{l = 1}^m{J_l} = \begin{cases}
		\frac{q^{-\binom{m}{2}}}{\grpIndex{\GL{m}{\residueField}}{\residueFieldUnipotentSubgroup{m}}} \left(\uniformizerValueForTwistedGamma q^{-s}\right)^{-\left(m - 2\right)} \cdot q \cdot \frac{1 - \uniformizerValueForTwistedGamma^{-1} q^{-\left(1-s\right)}}{\left(1 - \uniformizerValueForTwistedGamma q^{-s}\right)\left(q-1\right)}  & \left(\multiplicativeCharacter \cdot \unitaryDetCharacter^m \right)\restriction_{\multiplicativegroup{\integersring}} = 1,\\
		\frac{ q^{-\binom{m}{2}}  }{\grpIndex{\GL{m}{\residueField}}{\residueFieldUnipotentSubgroup{m}}} \left( \uniformizerValueForTwistedGamma q^{-s} \right)^{-\left(m-1\right)} \frac{1}{\sizeof{\multiplicativegroup{\residueField}}} \sum_{\lambda \in \multiplicativegroup{\residueField}}{ \multiplicativeCharacter\left(\lambda\right) \unitaryDetCharacter\left(\lambda\right)^m \fieldCharacter\left(-\lambda\right)} & \text{otherwise}
		\end{cases}.
	\end{equation}
	
	We now move to compute $\tilde{J} = \TwistedDualJSOfLocalFieldRepresentation{s}{\localFieldRepresentation\left( \evenPermutationMatrix \right)^{-1}W}{\phi}$. Following the same steps as before for the expression in \Cref{prop:js-dual-formulas}, we have $$\tilde{J} = \sum_{l=1}^m{\tilde{J}_l},$$ where
	$$ \tilde{J}_l =  q^{-\binom{m}{2}} q^{-l \left(s-1\right)} \uniformizerValueForTwistedGamma^l \cdot \int_{\proUnipotentRadical_m} \int_{\multiplicativegroup{\localField}} \multiplicativeCharacter\left(\lambda\right) \unitaryDetCharacter\left(\lambda\right)^m \abs{\lambda}^{m \left(s - 1\right)} \FourierTransformWithRespectToCharacter{\phi}{\fieldCharacter}\left(\lambda^{-1} \firstrowvector \uniformizerDiagonalMatrix{l}^{-1} \antidiagPermutationMatrix{l} \inverseTranspose{k} \right) \multiplicativeMeasure{\lambda} \multiplicativeMeasure{k}.$$
	Recall that $$\FourierTransformWithRespectToCharacter{\phi}{\fieldCharacter}\left(x\right) = \begin{cases}
	q^{\frac{m}{2}} \indicatorFunction{\firstrowvector}\left(\quotientMap\left(x\right)\right) & x \in \integersring^m,\\
	0 & \text{otherwise}
	\end{cases}.$$
	We have that $$\lambda^{-1} \firstrowvector \uniformizerDiagonalMatrix{l}^{-1} \antidiagPermutationMatrix{l} \inverseTranspose{k} = \begin{cases}
	\lambda^{-1} \rowvector{l+1} \inverseTranspose{k} &  1 \le l \le m-1 \\
	\lambda^{-1} \left(\modifiedUniformizer\right)^{-1} \firstrowvector \inverseTranspose{k} & l = m
	\end{cases}.$$ If $1 \le l \le m-1$, we have that $\lambda^{-1} \firstrowvector \uniformizerDiagonalMatrix{l}^{-1} \antidiagPermutationMatrix{l} \inverseTranspose{k}$ is a scalar multiple of $\rowvector{l + 1} \inverseTranspose{k} \in \integersring^m$ and the $l+1$ coordinate of $\rowvector{l + 1} \inverseTranspose{k}$ is $1$ modulo $\maximalideal$, and therefore  $\lambda^{-1} \firstrowvector \uniformizerDiagonalMatrix{l}^{-1} \antidiagPermutationMatrix{l} \inverseTranspose{k}$ can not be in the support of $\FourierTransformWithRespectToCharacter{\phi}{\fieldCharacter}$ for any scalar $\lambda$. If $l=m$, $\lambda^{-1} \firstrowvector \uniformizerDiagonalMatrix{m}^{-1} \antidiagPermutationMatrix{l} \inverseTranspose{k}$ is a scalar multiple of $\rowvector{1} \inverseTranspose{k} \in \integersring^m$, which satisfies $\quotientMap\left(\rowvector{1} \inverseTranspose{k}\right) \equiv \firstrowvector \left(\modulooperator \maximalideal \right)$. Therefore, in order for $\lambda^{-1} \firstrowvector \uniformizerDiagonalMatrix{l}^{-1} \antidiagPermutationMatrix{l} \inverseTranspose{k}$ to be in $\integersring^m$ and to be $\firstrowvector$ modulo $\maximalideal$, we must have $l=m$, and $\lambda^{-1} \left(\modifiedUniformizer\right)^{-1} \in 1 + \maximalideal$. Hence we have $\tilde{J}_l = 0$ for $1 \le l \le m-1$, and for $l = m$, we have that $\lambda$ is integrated on $\left(\modifiedUniformizer\right)^{-1} \left(1 + \maximalideal \right)$ and that \begin{align*}
		\tilde{J}_m &= q^{-\binom{m}{2}} q^{-m \left(s-1\right)} \uniformizerValueForTwistedGamma^m \cdot \int_{\proUnipotentRadical_m} \int_{1 + \maximalideal} \multiplicativeCharacter\left( \modifiedUniformizer \right)^{-1} \unitaryDetCharacter\left(\modifiedUniformizer\right)^{-m} \abs{ \left(\modifiedUniformizer\right) }^{-m \left(s - 1\right)} q^{\frac{m}{2}} \multiplicativeMeasure{\lambda} \multiplicativeMeasure{k}\\
		&= \frac{1}{\sizeof{\multiplicativegroup{\residueField}}} \frac{q^{-\binom{m}{2}}}{\grpIndex{\GL{m}{\residueField}}{\residueFieldUnipotentSubgroup{m}}} \cdot q^{\frac{m}{2}}.
	\end{align*}
	Therefore \begin{equation}\label{eq:computation-of-J-tilde-even-case}
		\tilde{J} = \tilde{J}_m = \frac{1}{\sizeof{\multiplicativegroup{\residueField}}} \frac{q^{-\binom{m}{2}}}{\grpIndex{\GL{m}{\residueField}}{\residueFieldUnipotentSubgroup{m}}} \cdot q^{\frac{m}{2}}.
	\end{equation}
	Recalling the fact that when $\left(\multiplicativeCharacter \cdot \unitaryDetCharacter^m\right)\restriction_{\multiplicativegroup{\integersring}} \ne 1$, the Gauss sum $$G\left(\multiplicativeCharacter \cdot \unitaryDetCharacter^m,\fieldCharacter\right) = \sum_{\lambda \in \multiplicativegroup{\residueField}}{ \multiplicativeCharacter\left(\lambda\right) \unitaryDetCharacter\left(\lambda\right)^m \fieldCharacter\left(-\lambda\right)}$$ has absolute value $\sqrt{q}$, we have that \begin{equation}\label{eq:inverse-of-gauss-sum}
		G\left(\multiplicativeCharacter \cdot \unitaryDetCharacter^m,\fieldCharacter\right)^{-1} = \frac{\conjugate{G\left(\multiplicativeCharacter \cdot \unitaryDetCharacter^m,\fieldCharacter\right)}}{q} = \frac{1}{q} \sum_{\lambda \in \multiplicativegroup{\residueField}}{ \multiplicativeCharacter\left(\lambda^{-1}\right) \unitaryDetCharacter\left(\lambda^{-m}\right) \fieldCharacter\left(\lambda\right)}.
	\end{equation} By \cref{eq:computation-of-J-even-case}, \cref{eq:computation-of-J-tilde-even-case}, and \cref{eq:inverse-of-gauss-sum} we get $$\twistedGammaFactorOfLocalField{s} = \frac{\tilde{J}}{J} = \begin{cases}
	\left(\uniformizerValueForTwistedGamma q^{-\left(s - \frac{1}{2}\right)}\right)^{m - 2} \cdot \frac{1 - \uniformizerValueForTwistedGamma q^{-s}}{1 - \uniformizerValueForTwistedGamma^{-1} q^{-\left(1-s\right)}} & \left(\multiplicativeCharacter \cdot \unitaryDetCharacter^m\right)\restriction_{\multiplicativegroup{\integersring}} = 1, \\
	\left(\uniformizerValueForTwistedGamma q^{-\left(s - \frac{1}{2}\right)}\right)^{m-1} \frac{1}{\sqrt{q}}\sum_{\lambda \in \multiplicativegroup{\residueField}}{ \multiplicativeCharacter\left(\lambda^{-1} \right) \unitaryDetCharacter\left(\lambda^{-m}\right) \fieldCharacter\left(\lambda\right)} & \text{otherwise}
	\end{cases}.$$
	The formula $\twistedGammaFactorOfLocalField{s} = \left(\uniformizerValueForTwistedGamma q^{-\left(s - \frac{1}{2}\right)}\right)^{m-1} \gammaFactorOfCharacter{s}{\exteriorCharacter{\left(\multiplicativeCharacter \cdot \unitaryDetCharacter^m\right)}}$ now follows from a standard computation of the local factors of Tate's functional equation, see for instance \cite[Section 7.1]{ramakrishnan2013fourier} or \cite[Proposition 3.8]{kudla2004tate}.
	
	The claim about the other twisted exterior square factors now follows from \Cref{lem:computation-of-exterior-square-factors} and the fact that $1 - \uniformizerValueForTwistedGamma Z$ and $1 - \uniformizerValueForTwistedGamma^{-1}q^{-1} Z^{-1}$ don't have mutual roots.
\end{proof}

\begin{rem}
	For the choice of test data $\left(\localFieldRepresentation \left(\evenPermutationMatrix^{-1}\right)W, \phi\right)$ as in the proof, we have that $\TwistedJSOfLocalFieldRepresentation{s}{\localFieldRepresentation \left(\evenPermutationMatrix^{-1}\right)W}{\phi}$ is non-zero if and only if $\unitaryDetCharacter$ is tamely ramified: otherwise on the right hand side of \cref{eq:evaluation-of-unitary-det-character} we will have a product $\prod_{i = 1}^m{ \unitaryDetCharacter\left( u'_i\right)} $, and since we integrate $u'_i$ over $1+\maximalideal$, we have that the integral vanishes unless the restriction of $\unitaryDetCharacter$ to $1+\maximalideal$ is trivial.
\end{rem}

\subsection{The odd case}\label{subsection:odd-computation}
In this section, we compute the twisted exterior square factors for the odd case $n = 2m + 1$.

\begin{thm}\label{thm:exterior-square-gamma-factor-for-odd-case}Let $\localFieldRepresentation$ be a simple supercuspidal representation of $\GL{2m+1}{\localField}$ with central character $\multiplicativeCharacter$, associated with the data $\left(t_0, \rootOfUniformizer\right)$. 
	Let $\unitaryDetCharacter : \multiplicativegroup{\localField} \rightarrow \multiplicativegroup{\cComplex}$ be a unitary tamely ramified character, i.e. $\unitaryDetCharacter \restriction_{1 + \maximalideal} = 1$. Then
	$$ \twistedGammaFactorOfLocalField{s} = \left(\unitaryDetCharacter\left( \modifiedUniformizer \right) \rootOfUniformizer^{2} q^{-\left(s - \frac{1}{2}\right)}\right)^m.$$
	Furthermore, in this case $\TwistedExteriorSquareLFunction{s}{\localFieldRepresentation}{\unitaryDetCharacter} = 1$, $\TwistedEpsilonFactorOfLocalField{s} = \twistedGammaFactorOfLocalField{s}$.
\end{thm}

\begin{proof}
	We compute $\TwistedJSOfLocalFieldRepresentation{s}{\localFieldRepresentation \left( \oddPermutationMatrix^{-1} \right) W}{\phi}$ and $\TwistedDualJSOfLocalFieldRepresentation{s}{\localFieldRepresentation \left( \oddPermutationMatrix^{-1} \right) W}{\phi}$, where again $W$ is the Whittaker function introduced in \Cref{subsubsection:simple-supercuspidal}, but this time \begin{align*}
		\phi \left(x\right) = \begin{cases}
			\indicatorFunction{0}\left(\quotientMap \left(x\right)\right) & x \in \integersring^m, \\
			0                                                             & \text{otherwise}
		\end{cases},
	& &
	\FourierTransformWithRespectToCharacter{\phi}{\fieldCharacter}\left(x\right) = \begin{cases}
		q^{-\frac{m}{2}} & x \in \integersring^m, \\
		0                & \text{otherwise}
	\end{cases},
	\end{align*} where $\indicatorFunction{0}$ is the indicator function of $0 \in \residueField^m$.

	By the Iwasawa decomposition,  $\TwistedJSOfLocalFieldRepresentation{s}{\localFieldRepresentation \left( \oddPermutationMatrix^{-1} \right)W}{\phi}$ is given by
	\begin{equation}\label{eq:iwasawa-decomposition-for-J-odd-case}
		\begin{split}
		J = \int & W \left(\oddPermutationMatrix \ShalikaUnipotentElementOdd{X} \ShalikaDiagonalElementOdd{ak} \ShalikaLowerUnipotentElementOdd{Z} \oddPermutationMatrix^{-1} \right) \phi \left(Z\right) \\
		& \cdot \abs{\det a}^{s-1} \unitaryDetCharacter \left(\det \left(ak\right)\right) \rightHaarMeasureModulus{\borelSubgroup_m}\left(a\right) dX \multiplicativeMeasure{a} \multiplicativeMeasure{k} dZ,
		\end{split}
	\end{equation}
	where $X$ is integrated on $\NilpotentLowerTriangular_m$, $a$ is integrated on the diagonal matrix subgroup $\diagonalSubgroup_m$, $k$ is integrated on $\maximalCompactSubgroup_m = \GL{m}{\integersring}$, and $Z$ is integrated on $\Mat{1}{m}{\localField}$. In order for $Z$ to be in the support of $\phi$, we must have $Z \in \Mat{1}{m}{\integersring}$, such that $\quotientMap \left(Z\right) = 0$, i.e. $Z \in \Mat{1}{m}{\maximalideal}$. For such fixed $Z$, we have that $$\quotientMap\left(\oddPermutationMatrix \ShalikaLowerUnipotentElementOdd{Z} \oddPermutationMatrix^{-1}\right) = \IdentityMatrix{2m + 1},$$ and therefore $$\oddPermutationMatrix \ShalikaLowerUnipotentElementOdd{Z} \oddPermutationMatrix^{-1} \in \proUnipotentRadical_{2m+1}.$$ Hence, in order for $X$, $a$, $k$ to contribute to the integral, we must have $$\oddPermutationMatrix \ShalikaUnipotentElementOdd{X} \ShalikaDiagonalElementOdd{ak} \oddPermutationMatrix^{-1} = \lambda u' \cyclicGroupGenerator{2m+1}^l k',$$ where $\lambda \in \multiplicativegroup{\localField}$, $u' \in \UnipotentSubgroup_{2m + 1}$, $l \in \zIntegers$, $k' \in \proUnipotentRadical_{2m+1}$. Since $\cyclicGroupGenerator{2m+1}^{2m + 1} = \modifiedUniformizer \IdentityMatrix{2m + 1}$, we may assume (by modifying $\lambda$), that $1 \le l \le 2m + 1$. Notice that \begin{equation}\label{eq:last-row-has-last-standard-row-vector}\lambda \cyclicGroupGenerator{2m+1}^l k' = u'^{-1} \oddPermutationMatrix \ShalikaUnipotentElementOdd{X} \ShalikaDiagonalElementOdd{ak} \oddPermutationMatrix^{-1},
	\end{equation} and the right hand side of \cref{eq:last-row-has-last-standard-row-vector} has $\rowvector{2m + 1} = \left(0,\dots,0,1\right)$ as its last row. On the other hand, the last row of $\lambda \cyclicGroupGenerator{2m+1}^l k'$ is $\lambda \modifiedUniformizer \rowvector{l} k'$, where $\rowvector{l}$ is the $l$-th standard row vector. Since $\rowvector{l} k'$ is the $l$-th row of $k'$, we have that $\quotientMap \left( \rowvector{l} k'\right)$ is the $l$-th row of an upper triangular unipotent matrix, and therefore the equality $\lambda \modifiedUniformizer \rowvector{l} k' = \rowvector{2m + 1}$ can't hold unless $l = 2m+1$. Thus we have $l = 2m+1$, and that the last row of $k'$ is a scalar multiple of $\rowvector{2m + 1}$. Since $k' \in \GL{2m + 1}{\integersring}$, we may assume (by modifying $\lambda$ by a unit) that the last row of $k'$ is $\rowvector{2m + 1}$.
	We write $k' = \left(\begin{smallmatrix}
				k'' & v\\
				& 1
			\end{smallmatrix}\right)$, where $k'' \in \proUnipotentRadical_{2m}$ and $v$ is a column vector in $\Mat{2m}{1}{\integersring}$. Writing $k' = \left( \begin{smallmatrix}
				\IdentityMatrix{2m} & v \\
				& 1
			\end{smallmatrix} \right) \smallDiagTwo{k''}{1}$, we may assume (by modifying $u'$) that $k' = \smallDiagTwo{k''}{1}$, which implies that $u' = \smallDiagTwo{u''}{1}$ for $u'' \in \UnipotentSubgroup_{2m}$. Thus we get that $\lambda \modifiedUniformizer = 1$, and that \begin{equation}\label{eq:assumption-of-lemma-of-support-is-satisfied-odd}
				\evenPermutationMatrix \ShalikaUnipotentElement{X} \ShalikaDiagonalElement{ak} \evenPermutationMatrix^{-1} = \lambda u'' \cyclicGroupGenerator{2m}^{2m} k'' \in \lambda \UnipotentSubgroup_{2m} \cyclicGroupGenerator{2m}^{2m} \proUnipotentRadical_{2m}.
			\end{equation}
			Since \cref{eq:assumption-of-lemma-of-support-is-satisfied-odd} holds, we can apply \Cref{lem:support-of-integrand} and use $\lambda \modifiedUniformizer = 1$ to get that \begin{equation}\label{eq:parameterization-of-a-odd-case}
				\begin{aligned}
				a = \lambda \cdot \uniformizerDiagonalMatrix{m} \cdot \diag \left(u_1, \dots, u_m\right) &= \diag \left(u_1, \dots, u_m\right), \text{where } u_1, \dots, u_m \in \multiplicativegroup{\integersring},\\
				\multiplicativeMeasure{a} &= \prod_{i=1}^m \multiplicativeMeasure{u_i},\\
				\rightHaarMeasureModulus{\borelSubgroup_m}\left(a\right) &= 1.
				\end{aligned}
			\end{equation} Denote \begin{equation}\begin{split}
			\label{eq:def-of-k0-odd-case-computation-of-J}
			k_0 &= \diag\left(u_1,\dots,u_m\right) k,\\			
			\multiplicativeMeasure{k} &= \multiplicativeMeasure{k_0}.
			\end{split}
			\end{equation}
			Then $k = \diag\left(u_1,\dots,u_m\right)^{-1} k_0$, and by \Cref{lem:support-of-integrand} \begin{equation}\label{eq:support-of-k0-J-odd-case}
				k_0 \in \quotientMap^{-1}\left(\residueFieldUnipotentSubgroupDoubleCoset{m}{m}\right) = \quotientMap^{-1}\left(\residueFieldUnipotentSubgroup{m}\right) = \proUnipotentRadical_m
			\end{equation} Furthermore, since $\quotientMap\left(k_0\right) \in \residueFieldUnipotentSubgroup{m} = \antidiagPermutationMatrix{m} \residueFieldUnipotentSubgroup{m}$, by \Cref{lem:support-of-integrand} we have that $X \in \NilpotentLowerTriangular_m\left(\integersring\right)$ and that $X$ satisfies $\quotientMap \left(X\right) \in \NilpotentLowerTriangular_m\left(\residueField\right) \cap \left(\antidiagPermutationMatrix{m} \NilpotentUpperTriangular_m\left(\residueField\right) \antidiagPermutationMatrix{m}^{-1} \right) = \left\{ 0_m \right\}$, i.e. \begin{equation}\label{eq:support-of-X-J-odd-case}
				X \in \NilpotentLowerTriangular_m\left(\maximalideal\right).
			\end{equation} Also, since $\quotientMap\left(k_0\right) \in \antidiagPermutationMatrix{m} \residueFieldUnipotentSubgroup{m}$, by \Cref{lem:support-of-integrand} we have for such $Z$, $X$, $a$, $k$ that $$\evenPermutationMatrix \ShalikaUnipotentElement{X} \ShalikaDiagonalElement{ak} \evenPermutationMatrix^{-1} = v_0,$$ where $\quotientMap \left( v_0 \right) = v_0'$ is an upper triangular unipotent matrix having zeros right above its diagonal, and therefore \begin{equation}\label{eq:definition-of-Y-J-odd-case}
				\begin{aligned}
				Y &= \oddPermutationMatrix \ShalikaUnipotentElementOdd{X} \ShalikaDiagonalElementOdd{ak} \ShalikaLowerUnipotentElementOdd{Z} \oddPermutationMatrix^{-1} \\
				&= \diagTwo{v_0}{1} \oddPermutationMatrix \ShalikaLowerUnipotentElementOdd{Z} \oddPermutationMatrix^{-1}
			\end{aligned}
			\end{equation} satisfies $\quotientMap \left(Y\right) = \smallDiagTwo{v_0'}{1}$, which implies that $\quotientMap \left(Y\right)$ is an upper unipotent matrix with zeros right above its diagonal. Since $Y$ also has zero at its left bottom corner, we have from \cref{eq:simple-supercuspidal-whittaker-function} \begin{equation}\label{eq:whittaker-value-for-J-odd-case}
				W \left( Y \right) = 1.
			\end{equation}
		
		From \cref{eq:parameterization-of-a-odd-case} and \cref{eq:def-of-k0-odd-case-computation-of-J}, we have that $ak = k_0 \in \proUnipotentRadical_m$, so $\det\left(ak\right) = \det\left(k_0\right) \in 1 + \maximalideal$, which implies \begin{equation}\label{eq:evaluation-of-determinant-J-odd-case}
			\begin{aligned}\\
			\abs{\det a}^s &= 1,\\
			\unitaryDetCharacter\left(\det \left(ak\right)\right) &= 1,
			\end{aligned}
		\end{equation} as $\unitaryDetCharacter$ is tamely ramified.
		
	Therefore, we have by substituting in \cref{eq:iwasawa-decomposition-for-J-odd-case} the equations (\ref{eq:parameterization-of-a-odd-case})-(\ref{eq:evaluation-of-determinant-J-odd-case}) that \begin{equation}\label{eq:final-evaluation-of-J-odd-case}
		\begin{split}
		\TwistedJSOfLocalFieldRepresentation{s}{\localFieldRepresentation \left(\oddPermutationMatrix^{-1}\right)W}{\phi} &= \int_{\Mat{1}{m}{\maximalideal}} \int_{\proUnipotentRadical_m} \int_{\left(\multiplicativegroup{\integersring}\right)^m} \int_{\NilpotentLowerTriangular_m\left(\maximalideal\right)} dX \left(\prod_{i=1}^m \multiplicativeMeasure{u_i}\right) \multiplicativeMeasure{k_0} dZ \\
		&=  \frac{1}{\sizeof{\Mat{1}{m}{\residueField}}} \frac{1}{\grpIndex{\GL{m}{\residueField}}{\residueFieldUnipotentSubgroup{m}}} \frac{1}{\sizeof{\NilpotentLowerTriangular_m \left(\residueField\right)}}.
		\end{split}
	\end{equation}

	We now move to compute $\TwistedDualJSOfLocalFieldRepresentation{s}{\localFieldRepresentation \left(\oddPermutationMatrix^{-1}\right)W}{\phi}$. By \Cref{prop:js-dual-formulas}, we need to evaluate the integral
	\begin{equation}\label{eq:iwasawa-decomposition-for-J-tilde-odd-case}
		\begin{split}
		\tilde{J} = \int & W \left(\antiDiagTwo{1}{\IdentityMatrix{2m}} \oddPermutationMatrix \ShalikaUnipotentElementOdd{X} \ShalikaDiagonalElementOdd{ak} \ShalikaUpperRightUnipotentElementOdd{-\transpose{Z}} \oddPermutationMatrix^{-1} \right) \\
		& \cdot \FourierTransformWithRespectToCharacter{\phi}{\fieldCharacter} \left(Z\right)\abs{\det a}^{s} \unitaryDetCharacter \left(\det\left(ak\right)\right) \rightHaarMeasureModulus{\borelSubgroup_m}\left(a\right) dX \multiplicativeMeasure{a} \multiplicativeMeasure{k} dZ,
		\end{split}
	\end{equation}
	where again $X \in \NilpotentLowerTriangular_m$, $a \in \diagonalSubgroup_m$, $k \in \maximalCompactSubgroup_m = \GL{m}{\integersring}$, $Z \in \Mat{1}{m}{\localField}$. We notice that for every $Z \in \Mat{1}{m}{\integersring}$, $$\oddPermutationMatrix \ShalikaUpperRightUnipotentElementOdd{-\transpose{Z}} \oddPermutationMatrix^{-1} \in \proUnipotentRadical_{2m+1}.$$ Therefore, in order for $X$, $a$, $k$ to support the integrand, we need \begin{equation}\label{eq:last-column-of-the-left-hand-side-is-one}
		\antiDiagTwo{1}{\IdentityMatrix{2m}} \oddPermutationMatrix \ShalikaUnipotentElementOdd{X} \ShalikaDiagonalElementOdd{ak} \oddPermutationMatrix^{-1} = \lambda u' \cyclicGroupGenerator{2m + 1}^l k',
	\end{equation} where $\lambda \in \multiplicativegroup{\localField}$, $u' \in \UnipotentSubgroup_{2m + 1}$, $l \in \zIntegers$, $k' \in \proUnipotentRadical_{2m+1}$. Note that the left hand side of \cref{eq:last-column-of-the-left-hand-side-is-one} has $\standardColumnVector{1} = \transpose{\left(1,0,\dots,0\right)}$ as its last column. Therefore, $\lambda u' \cyclicGroupGenerator{2m + 1}^l k'$ needs to have $\standardColumnVector{1}$ as its last column, i.e. $\lambda u' \cyclicGroupGenerator{2m + 1}^l k' \standardColumnVector{2m+1} = \standardColumnVector{1}$, which implies $\lambda \cyclicGroupGenerator{2m + 1}^l k' \standardColumnVector{2m+1} = {u'}^{-1} \standardColumnVector{1} = \standardColumnVector{1}$. Since the $2m+1-l$ coordinate of $\quotientMap \left(\cyclicGroupGenerator{2m + 1}^l k' \standardColumnVector{2m+1} \right)$ is 1, we must have $l = 2m$. Therefore, from $\lambda \cyclicGroupGenerator{2m + 1}^{2m} k' \standardColumnVector{2m+1} = \standardColumnVector{1}$ we get that the last column of $k'$ is a scalar multiple of $\standardColumnVector{2m + 1}$. Modifying $\lambda$ by a unit, we may assume that $k'$ has $\standardColumnVector{2m + 1}$ as its last column. Write $k' = \left( \begin{smallmatrix}
				k'' & \\
				v & 1
			\end{smallmatrix} \right)$, where $k'' \in \proUnipotentRadical_{2m}$, $v \in \Mat{1}{2m}{\maximalideal}$. Writing $\cyclicGroupGenerator{2m + 1}^{2m} = \smallAntiDiagTwo{1}{\modifiedUniformizer\IdentityMatrix{2m}}$, we have
            \begin{equation*}
            \begin{split}
            \lambda u' \cyclicGroupGenerator{2m+1}^{2m} k' &= \lambda u' \antiDiagTwo{1}{\modifiedUniformizer \IdentityMatrix{2m}} \begin{pmatrix}
			k'' &\\
			v & 1
			\end{pmatrix}\\
            &=  \lambda u' \begin{pmatrix}
			1 & v\left(\modifiedUniformizer k''\right)^{-1}\\
			& \IdentityMatrix{2m}
			\end{pmatrix} \antiDiagTwo{1}{\modifiedUniformizer \IdentityMatrix{2m}} \diagTwo{k''}{1}.
            \end{split}
            \end{equation*}
            Therefore, by replacing $u'$ by $u' \left( \begin{smallmatrix}
				1 & {v {\left(\modifiedUniformizer k''\right)}^{-1}}\\
				  & \IdentityMatrix{2m}
			\end{smallmatrix} \right)$, we may assume $k' = \smallDiagTwo{k''}{1}$, which implies by \cref{eq:last-column-of-the-left-hand-side-is-one} that $u' = \smallDiagTwo{1}{u''}$ for $u'' \in \UnipotentSubgroup_{2m}$. Substituting the expressions for $u'$, $k'$, and the expression $\cyclicGroupGenerator{2m+1}^{2m} = \smallAntiDiagTwo{1}{\modifiedUniformizer \IdentityMatrix{2m}}$ in \cref{eq:last-column-of-the-left-hand-side-is-one}, we get that $$\oddPermutationMatrix \ShalikaUnipotentElementOdd{X} \ShalikaDiagonalElementOdd{ak} \oddPermutationMatrix^{-1} = \diagTwo{\lambda \modifiedUniformizer u'' k''}{\lambda},$$ and therefore $\lambda = 1$ and $$\evenPermutationMatrix \ShalikaUnipotentElement{X} \ShalikaDiagonalElement{ak} \evenPermutationMatrix^{-1} = \modifiedUniformizer u'' k''.$$
	By \Cref{lem:support-of-integrand}, we have that \begin{equation}\label{eq:parameterization-of-a-J-tilde-odd-case}
		\begin{split}
		a = \uniformizerDiagonalMatrix{m} \cdot \diag \left(u_1, \dots, u_m \right) &= \modifiedUniformizer \diag \left(u_1, \dots, u_m \right), \text{ where } u_1, \dots, u_m \in \multiplicativegroup{\integersring},\\
		\multiplicativeMeasure{a} &= \prod_{i=1}^m{\multiplicativeMeasure{u_i}},\\
		\rightHaarMeasureModulus{\borelSubgroup_m}\left(a\right) &= 1.
		\end{split}
	\end{equation}
	
	Denote \begin{equation}\label{eq:def-of-k0-odd-case-computation-of-J-tilde}
		\begin{split}
		k_0 &= \diag \left(u_1, \dots, u_m\right) k,\\
		\multiplicativeMeasure{k_0} &= \multiplicativeMeasure{k}.
		\end{split}
	\end{equation} Then by \Cref{lem:support-of-integrand}, \begin{equation}
		k_0 \in \quotientMap^{-1}\left( \residueFieldUnipotentSubgroupDoubleCoset{m}{m}\right) = \quotientMap^{-1}\left(\residueFieldUnipotentSubgroup{m}\right) = \proUnipotentRadical_m.
	\end{equation}
	Since $\quotientMap\left(k_0\right) \in \antidiagPermutationMatrix{m} \residueFieldUnipotentSubgroup{m} = \residueFieldUnipotentSubgroup{m}$, by \Cref{lem:support-of-integrand} we have $X \in \NilpotentLowerTriangular_m\left(\integersring\right)$ satisfies $\quotientMap \left(X\right) \in \NilpotentLowerTriangular_m\left(\residueField\right)\cap \left(\antidiagPermutationMatrix{m} \NilpotentUpperTriangular_m\left(\residueField\right) \antidiagPermutationMatrix{m}^{-1} \right) = \left\{0_m\right\}$, i.e. \begin{equation}\label{eq:support-of-X-J-tilde-odd-case}
		X \in \NilpotentLowerTriangular_m\left(\maximalideal\right).
	\end{equation}
	Also, since $\quotientMap\left(k_0\right) \in \antidiagPermutationMatrix{m} \residueFieldUnipotentSubgroup{m}$, by \Cref{lem:support-of-integrand} we have for such elements that $$\evenPermutationMatrix \ShalikaUnipotentElement{X} \ShalikaDiagonalElement{ak} \evenPermutationMatrix^{-1} = \cyclicGroupGenerator{2m}^{2m} v_0 = \modifiedUniformizer v_0,$$ where $v_0 \in \GL{2m}{\integersring}$ satisfies that $\quotientMap\left(v_0\right) = v_0'$ is an upper triangular unipotent matrix with zeros right above its diagonal. Hence, we have that \begin{equation}\label{eq:definition-of-Y-J-tilde-odd-case}
		\begin{split}
		Y & = \antiDiagTwo{1}{\IdentityMatrix{2m}} \oddPermutationMatrix \ShalikaUnipotentElementOdd{X} \ShalikaDiagonalElementOdd{ak} \ShalikaUpperRightUnipotentElementOdd{-\transpose{Z}} \oddPermutationMatrix^{-1} \\
		%			& = \antiDiagTwo{1}{\IdentityMatrix{2m}} \diagTwo{\modifiedUniformizer v_0}{1}  \oddPermutationMatrix \ShalikaUpperRightUnipotentElementOdd{-\transpose{Z}} \oddPermutationMatrix^{-1} \\
		%			& = \antiDiagTwo{1}{\modifiedUniformizer \IdentityMatrix{2m}} \diagTwo{v_0}{1}  \oddPermutationMatrix \ShalikaUpperRightUnipotentElementOdd{-\transpose{Z}} \oddPermutationMatrix^{-1}\\
		& = \cyclicGroupGenerator{2m+1}^{2m} \diagTwo{v_0}{1}  \oddPermutationMatrix \ShalikaUpperRightUnipotentElementOdd{-\transpose{Z}} \oddPermutationMatrix^{-1}.
		\end{split}
	\end{equation}
	Denote $$Y' = \diagTwo{v_0}{1} \oddPermutationMatrix \ShalikaUpperRightUnipotentElementOdd{-\transpose{Z}} \oddPermutationMatrix^{-1},$$ then $Y = \cyclicGroupGenerator{2m + 1}^{2m} Y'$, and $\quotientMap \left(Y'\right)$ is an upper triangular unipotent matrix with zeros right above its diagonal, as it is a product of such. $Y'$ also has zero at its left bottom corner. Therefore by \cref{eq:simple-supercuspidal-whittaker-function} \begin{equation}\label{eq:whittaker-value-for-J-tilde-odd-case}
		W \left(Y\right) = \rootOfUniformizer^{2m}.
	\end{equation}
	By \cref{eq:parameterization-of-a-J-tilde-odd-case} and \cref{eq:def-of-k0-odd-case-computation-of-J-tilde}, we have that $ak = \modifiedUniformizer k_0$, and therefore $\det\left(ak\right) = \left(\modifiedUniformizer\right)^m \det k_0$. Since $k_0 \in \proUnipotentRadical_m$, then $\det k_0 \in 1 + \maximalideal$, which implies \begin{equation}\label{eq:evaluation-of-determinant-J-tilde-odd-case}
		\begin{split}
		\abs{\det a}^s &= q^{-ms},\\		
		\unitaryDetCharacter\left(\det\left(ak\right)\right) &= \unitaryDetCharacter\left(\modifiedUniformizer\right)^m,
		\end{split}
	\end{equation}
	as $\unitaryDetCharacter$ is tamely ramified.
	
	By substituting in \cref{eq:iwasawa-decomposition-for-J-tilde-odd-case} the equations (\ref{eq:parameterization-of-a-J-tilde-odd-case})-(\ref{eq:evaluation-of-determinant-J-tilde-odd-case}), we have \begin{equation}\label{eq:final-evaluation-of-J-tilde-odd-case}
		\begin{split}
			\tilde{J} &= \int_{\Mat{1}{m}{\integersring}} \int_{\proUnipotentRadical_m} \int_{\left(\multiplicativegroup{\integersring}\right)^m} \int_{\NilpotentLowerTriangular_m\left( \maximalideal \right)} \rootOfUniformizer^{2m} q^{-\frac{m}{2}}  q^{-ms} \unitaryDetCharacter\left(\modifiedUniformizer\right)^m  dX \left(\prod_{i=1}^m\multiplicativeMeasure{u_i}\right) \multiplicativeMeasure{k_0} dZ \\
			& = \frac{1}{\grpIndex{\GL{m}{\residueField}}{\residueFieldUnipotentSubgroup{m}}} \frac{1}{\sizeof{\NilpotentLowerTriangular_m \left(\residueField\right)}} q^{-\frac{m}{2}} \rootOfUniformizer^{2m}  q^{-m s} \unitaryDetCharacter\left(\modifiedUniformizer\right)^m.
		\end{split}
	\end{equation}

	We get from \cref{eq:final-evaluation-of-J-odd-case} and \cref{eq:final-evaluation-of-J-tilde-odd-case}, $$\twistedGammaFactorOfLocalField{s} = \frac{\tilde{J}}{J} = \unitaryDetCharacter\left(\modifiedUniformizer\right)^m \rootOfUniformizer^{2m} q^{-ms} q^{\frac{m}{2}}.$$

	The result regarding the other local factors now follows from \Cref{thm:functional-equation-gamma-factor} and \Cref{thm:roots-of-l-function}.
\end{proof}

\section{Exterior square gamma factors local converse theorem}\label{section:local-converse-theorem}

In this section we present and prove a local converse theorem for simple supercuspidal representations. Unlike previous local converse theorems, which are usually based on Rankin-Selberg gamma factors, our theorem is based on twisted exterior square gamma factors.

\begin{thm} \label{thm:converse-theorem-for-simple-supercuspidal-data}
	Let $n = 2m$ or $n = 2m+1$. Let $\representationDeclaration{\localFieldRepresentation}$, $\representationDeclaration{\localFieldRepresentation'}$ be simple supercuspidal representations of $\GL{n}{\localField}$, with the same central character $\multiplicativeCharacter = \centralCharacter{\localFieldRepresentation} = \centralCharacter{\localFieldRepresentation'}$, such that $\localFieldRepresentation$, $\localFieldRepresentation'$ are associated with the data $\left(t_0, \rootOfUniformizer \right)$ and $\left(t_0' , \rootOfUniformizer'  \right)$ correspondingly, where $\rootOfUniformizer^n =  \multiplicativeCharacter\left(\modifiedUniformizer\right)$, $\rootOfUniformizer'^n =  \multiplicativeCharacter\left({t'}^{-1} \uniformizer \right)$, and $t,t' \in \multiplicativegroup{\integersring}$ are lifts of $t_0$, $t'_0$ respectively, i.e. $\quotientMap\left(t\right) = t_0$, $\quotientMap\left(t'\right) = t'_0$. Assume that  
	\begin{enumerate}
		\item If $n = 2m$, then $\gcd \left(m-1, q-1\right)=1$.
		\item If $n = 2m+1$, then $\gcd \left(m, q-1\right)=1$.
	\end{enumerate}
	Suppose that for every unitary tamely ramified character $\unitaryDetCharacter : \multiplicativegroup{\localField} \rightarrow \multiplicativegroup{\cComplex}$, we have \begin{equation}\label{eq:equality-of-twisted-gamma-factors}
		\twistedGammaFactorOfLocalField{s} = \twistedGammaFactor{s}{\localFieldRepresentation'}{\unitaryDetCharacter}.
	\end{equation} Then $\rootOfUniformizer = \pm \rootOfUniformizer'$ and $t_0 = t_0'$.
\end{thm}
\begin{proof}
	Suppose $n = 2m$. Let $\unitaryDetCharacter$ be a unitary tamely ramified character. Denote $\uniformizerValueForTwistedGamma = \rootOfUniformizer^2 \cdot \unitaryDetCharacter\left( \left(-1\right)^{m-1} \modifiedUniformizer \right)$, $\uniformizerValueForTwistedGamma' = \rootOfUniformizer'^2 \cdot \unitaryDetCharacter\left( \left(-1\right)^{m-1} t'^{-1} \uniformizer \right)$. We claim that $\uniformizerValueForTwistedGamma = \uniformizerValueForTwistedGamma'$.
	
	If $\left(\multiplicativeCharacter \cdot \unitaryDetCharacter^m \right)\restriction_{\multiplicativegroup{\integersring}} = 1$, we have by \Cref{thm:exterior-square-gamma-factor-for-even-case} and by \cref{eq:equality-of-twisted-gamma-factors} that \begin{equation}\label{eq:rational-functions-exterior-square-gamma-factors}
	 	\left(\uniformizerValueForTwistedGamma q^{-\left(s - \frac{1}{2}\right)}\right)^{m-2} \frac{1 - \uniformizerValueForTwistedGamma  q^{-s}}{1 - \uniformizerValueForTwistedGamma^{-1} q^{- \left(1-s\right)}} = \left(\uniformizerValueForTwistedGamma' q^{-\left(s - \frac{1}{2}\right)}\right)^{m-2} \frac{1 - \uniformizerValueForTwistedGamma'  q^{-s}}{1 - {\uniformizerValueForTwistedGamma'}^{-1} q^{- \left(1-s\right)}},
	 \end{equation}and we get $\xi = \xi'$ by comparing the poles of both sides of \cref{eq:rational-functions-exterior-square-gamma-factors}. 
	 
	 If $\left(\multiplicativeCharacter \cdot \unitaryDetCharacter^m \right)\restriction_{\multiplicativegroup{\integersring}} \ne 1$, we have by \Cref{thm:exterior-square-gamma-factor-for-even-case} and by \cref{eq:equality-of-twisted-gamma-factors} that
	 $$ \left(\uniformizerValueForTwistedGamma q^{- \left(s - \frac{1}{2}\right)} \right)^{m-1}  \frac{1}{\sqrt{q}} \sum_{\lambda \in \multiplicativegroup{\residueField}}{\fieldCharacter \left(\lambda\right) \multiplicativeCharacter \left(\lambda^{-1}\right) \unitaryDetCharacter \left(\lambda^{-m}\right)} = \left(\uniformizerValueForTwistedGamma' q^{- \left(s - \frac{1}{2}\right)} \right)^{m-1}  \frac{1}{\sqrt{q}} \sum_{\lambda \in \multiplicativegroup{\residueField}}{\fieldCharacter \left(\lambda\right) \multiplicativeCharacter \left(\lambda^{-1}\right) \unitaryDetCharacter \left(\lambda^{-m}\right)}.$$ Therefore \begin{equation}\label{eq:converse-thm-twisted-root-is-root-of-order-m-minus-1}
	 	\left(\frac{\uniformizerValueForTwistedGamma}{\uniformizerValueForTwistedGamma'}\right)^{m-1} = 1.
	 \end{equation}
	 On the other hand, \begin{equation}\label{eq:ratio-of-uniformizer-values-for-twisted-gamma-even}
	 	\frac{\uniformizerValueForTwistedGamma}{\uniformizerValueForTwistedGamma'} = \frac{\rootOfUniformizer^2 \unitaryDetCharacter \left( \left(-1\right)^{m-1} \modifiedUniformizer  \right)}{\rootOfUniformizer'^2 \unitaryDetCharacter \left( \left(-1\right)^{m-1} {t'}^{-1} \uniformizer \right)} = \frac{\rootOfUniformizer^2}{\rootOfUniformizer'^2} \unitaryDetCharacter\left( t' t^{-1} \right).
	 \end{equation}
	 Since $\rootOfUniformizer^{2m} = \multiplicativeCharacter\left(\modifiedUniformizer\right)$ and $\rootOfUniformizer'^{2m} = \multiplicativeCharacter\left( t'^{-1} \uniformizer \right)$, we get from \cref{eq:ratio-of-uniformizer-values-for-twisted-gamma-even} that $$\left(\frac{\uniformizerValueForTwistedGamma}{\uniformizerValueForTwistedGamma'}\right)^m = \frac{\multiplicativeCharacter\left(\modifiedUniformizer\right)}{\multiplicativeCharacter\left(t'^{-1} \uniformizer\right)} \unitaryDetCharacter\left(t' t^{-1}\right)^m = \multiplicativeCharacter\left(t' t^{-1}\right) \unitaryDetCharacter\left(t' t^{-1}\right)^m,$$
	 which implies that \begin{equation}\label{eq:converse-thm-twisted-root-is-root-of-order-m-times-q-minus-1}
	 	\left(\frac{\uniformizerValueForTwistedGamma}{\uniformizerValueForTwistedGamma'}\right)^{m \left(q - 1\right)} = 1,
	 \end{equation} as $\multiplicativeCharacter, \unitaryDetCharacter$ are tamely ramified. Since $m-1$ is coprime to $m$ and to $q-1$, we have that $\gcd\left(m-1,m\left(q-1\right)\right) = 1$, and therefore $m-1$ is invertible modulo $m \left(q - 1\right)$, which implies from \cref{eq:converse-thm-twisted-root-is-root-of-order-m-minus-1} and \cref{eq:converse-thm-twisted-root-is-root-of-order-m-times-q-minus-1} that $\frac{\xi}{\xi'} = 1$, i.e. $\uniformizerValueForTwistedGamma = \uniformizerValueForTwistedGamma'$.
	 
	 We proved $\uniformizerValueForTwistedGamma = \uniformizerValueForTwistedGamma'$. Therefore by \cref{eq:ratio-of-uniformizer-values-for-twisted-gamma-even}, we have $\unitaryDetCharacter\left( t' t^{-1} \right) = \left(\rootOfUniformizer' \rootOfUniformizer^{-1} \right)^2$ for every unitary tamely ramified character $\unitaryDetCharacter$. Choosing the trivial character, this implies that $\rootOfUniformizer^2 = \rootOfUniformizer'^2$. Suppose $t_0 \ne t_0'$, then there exists a unitary character $\unitaryDetCharacter_0 : \multiplicativegroup{\residueField} \rightarrow \multiplicativegroup{\cComplex}$, such that $\unitaryDetCharacter_0\left(t_0' t_0^{-1}\right) \ne 1$, and we can lift this to a unitary tamely ramified character $\unitaryDetCharacter : \multiplicativegroup{\localField} \rightarrow \multiplicativegroup{\cComplex}$ that satisfies $\unitaryDetCharacter \restriction_{\multiplicativegroup{\integersring}} = \unitaryDetCharacter_0 \circ \quotientMap$ and then $\unitaryDetCharacter\left(t' t^{-1}\right) = \unitaryDetCharacter_0\left(t_0' t_0^{-1}\right) \ne 1$, which is a contradiction to $\unitaryDetCharacter\left( t' t^{-1} \right) = \left(\rootOfUniformizer' \rootOfUniformizer^{-1} \right)^2 = 1$. Therefore, we must have $t_0 = t_0'$.
	 
	 For $n = 2m+1$, the proof is similar. We have that $$\left(\frac{\rootOfUniformizer}{\rootOfUniformizer'}\right)^{2m + 1} = \frac{\multiplicativeCharacter\left( \modifiedUniformizer \right)}{\multiplicativeCharacter \left( t'^{-1} \uniformizer \right)} = \multiplicativeCharacter\left( t' t^{-1} \right),$$ and therefore
	 \begin{equation}\label{eq:converse-thm-twisted-root-is-root-of-order-2m-plus-1-times-q-minus-1}
	 	\left(\frac{\rootOfUniformizer}{\rootOfUniformizer'}\right)^{\left(2m + 1\right)\left(q-1\right)} = 1,
	 \end{equation} as $\multiplicativeCharacter$ is tamely ramified.
	 By \cref{eq:equality-of-twisted-gamma-factors} and \Cref{thm:exterior-square-gamma-factor-for-odd-case}, we have for any unitary tamely ramified character $\mu$ $$\left(\unitaryDetCharacter\left( \modifiedUniformizer \right) \rootOfUniformizer^{2} q^{-\left(s - \frac{1}{2}\right)}\right)^m = \left(\unitaryDetCharacter\left( t'^{-1} \uniformizer \right) \rootOfUniformizer'^{2} q^{-\left(s - \frac{1}{2}\right)}\right)^m,$$ which implies that
	 \begin{equation}\label{eq:ratio-of-uniformizer-values-for-twisted-gamma-odd}
	 	\left(\frac{\rootOfUniformizer}{\rootOfUniformizer'}\right)^{2m} = \unitaryDetCharacter\left(t'^{-1} t\right)^m.
	 \end{equation}
	 Substituting the trivial character in \cref{eq:ratio-of-uniformizer-values-for-twisted-gamma-odd}, one gets $\rootOfUniformizer^{2m} = \rootOfUniformizer'^{2m}$, which implies that \begin{equation}\label{eq:converse-thm-twisted-root-is-root-of-order-m}
	 	\left(\frac{\rootOfUniformizer^2}{\rootOfUniformizer'^2}\right)^m = 1.
	 \end{equation} Since $\gcd\left(m, q - 1\right) = \gcd\left(m, 2m+1\right) = 1$, we get that $\gcd\left(m, \left(2m+1\right)\left(q-1\right)\right) = 1$, and therefore $m$ is invertible modulo $\left(2m+1\right)\left(q-1\right)$. By \cref{eq:converse-thm-twisted-root-is-root-of-order-2m-plus-1-times-q-minus-1},  \cref{eq:converse-thm-twisted-root-is-root-of-order-m}, this implies $\frac{\rootOfUniformizer^2}{\rootOfUniformizer'^2} = 1$. By \cref{eq:ratio-of-uniformizer-values-for-twisted-gamma-odd} and \cref{eq:converse-thm-twisted-root-is-root-of-order-m}, we have for every unitary tamely ramified character $\unitaryDetCharacter : \multiplicativegroup{\localField} \rightarrow \multiplicativegroup{\cComplex}$, $\unitaryDetCharacter \left( {t'}^{-1} t \right)^m = 1$. Since $\unitaryDetCharacter\left(t'^{-1} t\right)^{q-1} = 1$, and since $m$ is coprime to $q - 1$, $m$ is invertible modulo $q-1$, and therefore we have that for every unitary tamely ramified character $\unitaryDetCharacter$, $\unitaryDetCharacter\left({t'}^{-1} t\right) = 1$. As in the even case, this implies $t_0 = t_0'$.
\end{proof}

\begin{rem}
\begin{enumerate}
	\item In the even case, although we can not prove $\pi \cong \pi'$, we get $\zeta = \pm \zeta'$. On the other hand, if $\pi$ and $\pi'$ are simple supercuspidal representations with the same central character $\omega$ associated to the data $(t_0, \zeta)$ and $(t_0, -\zeta)$, then we must have $\twistedGammaFactorOfLocalField{s} = \twistedGammaFactor{s}{\localFieldRepresentation'}{\unitaryDetCharacter}$ for all tamely ramified character $\mu$. Because by \Cref{thm:exterior-square-gamma-factor-for-even-case},
	\begin{align*}
		\twistedGammaFactorOfLocalField{s} &= \left( \uniformizerValueForTwistedGamma q^{-\left(s - \frac{1}{2}\right)} \right)^{m - 1} \gammaFactorOfCharacter{s}{ \twistedMultiplicativeExteriorCharacter},\\
		\twistedGammaFactor{s}{\localFieldRepresentation'}{\unitaryDetCharacter} &= \left( \uniformizerValueForTwistedGamma' q^{-\left(s - \frac{1}{2}\right)} \right)^{m - 1} \gamma\left(s, \left(\omega \cdot \mu^m\right)_{\xi'}, \psi\right),
	\end{align*}
	and $\xi = \rootOfUniformizer^2 \cdot \unitaryDetCharacter\left( \left(-1\right)^{m-1} \modifiedUniformizer \right) = \xi'$.

	\item In the odd case, we actually get that $\rootOfUniformizer = \rootOfUniformizer'$, since $\rootOfUniformizer^2 = \rootOfUniformizer'^2$ and $\rootOfUniformizer^{2m+1} = \rootOfUniformizer'^{2m+1} = \multiplicativeCharacter\left( \modifiedUniformizer \right)$, and then $$\rootOfUniformizer = \frac{\rootOfUniformizer^{2m+1}}{\left(\rootOfUniformizer^2\right)^m} = \frac{\rootOfUniformizer'^{2m+1}}{\left(\rootOfUniformizer'^2\right)^m} = \rootOfUniformizer'.$$
As a consequence, when the hypotheses in \Cref{thm:converse-theorem-for-simple-supercuspidal-data} are met, we have $t_0 = t_0'$ and $\rootOfUniformizer = \rootOfUniformizer'$, so $\pi \cong \pi'$.
\end{enumerate}
\end{rem}

\subsection*{Acknowledgments}

The authors thank Zahi Hazan for introducing them to simple supercuspidal representations. The authors thank James Cogdell for his comments on an earlier version of this paper.

\bibliographystyle{amsalpha}
\bibliography{bibliography}

\end{document}